\documentclass[11pt]{article}

\usepackage{amsfonts}
\usepackage{amsthm}
\usepackage{amsmath}
\usepackage{amssymb}
\usepackage{amscd}
\usepackage{mathrsfs}
\usepackage[mathscr]{eucal}
\usepackage{graphicx}
\usepackage{epstopdf} 
\usepackage{pict2e}
\usepackage{epic}
\usepackage{xcolor}
\usepackage{float}
\usepackage{mathtools}
\usepackage{centernot}
\usepackage{bbm}

\usepackage{cite}
\usepackage{hyperref}
\hypersetup{colorlinks=true, urlcolor= black, linkcolor=black, citecolor=black}

\usepackage[utf8]{inputenc}
\usepackage[T1]{fontenc}
\usepackage{lmodern}
\usepackage{dsfont}
\usepackage{indentfirst}

\numberwithin{equation}{section}

\theoremstyle{plain}
\newtheorem{Thm}{Theorem}[section]
\newtheorem{Lem}[Thm]{Lemma}
\newtheorem{Coro}[Thm]{Corollary}
\newtheorem{Prop}[Thm]{Proposition}

\theoremstyle{definition}

\newtheorem{Rem}[Thm]{Remark}

\usepackage[margin=1 in]{geometry}

\newcommand\n{\mathbf{n}}
\newcommand\m{\mathbf{m}}

\newcommand\fg{\mathfrak{g}}
\newcommand{\connect}{\xleftrightarrow}

\newcommand{\Z}{\mathbb Z}
\newcommand{\N}{\mathbb N}

\newcommand{\lr}[4]{#3\xleftrightarrow[#1]{#2} #4}
\newcommand{\nlr}[4]{#3\mathrel{\mathop{\centernot\longleftrightarrow}_{#1}^{#2}} #4}

\title{Supercritical sharpness for Ising double random currents}

\begin{document}

\author{Trishen S. Gunaratnam\footnote{University College London, \url{trishen.gunaratnam@ucl.ac.uk}}\:, Christoforos Panagiotis\footnote{University of Bath, \url{cp2324@bath.ac.uk}}\:, Romain Panis\footnote{Université Lyon 1, 
\url{panis@math.univ-lyon1.fr}}\:, Franco Severo\footnote{CNRS and Sorbonne Université, \url{severo@lpsm.paris}}}
\date{}
\maketitle

\begin{abstract}
	We show that the percolation model induced by the double random current representation of the Ising model on $\mathbb Z^d$, with $d\geq 3$, satisfies local uniqueness of the macroscopic cluster with high probability for every $\beta>\beta_c$, uniformly in boundary conditions (i.e.~sources). Such a result is sometimes called supercritical sharpness in percolation theory and yields a detailed description of the model via standard renormalisation techniques. As applications, we show the exponential decay of truncated correlations for the Ising model and exponential mixing for the FK-Ising measure, results previously obtained by Duminil-Copin, Goswami and Raoufi. However, our proof is considerably simpler, as it replaces the use of a highly technical multi-valued map principle with a simple sprinkling argument. 
\end{abstract}

\section{Introduction}

The Ising model is one of the most studied models of statistical mechanics. We refer to \cite{FriedliVelenikIntroStatMech2017,DuminilLecturesOnIsingandPottsModels2019} for mathematical introductions and to \cite{DuminilreviewIsing} for a review. In dimensions $d\geq 2$, the model undergoes a \emph{phase transition} \cite{PeierlsIsing1936} at a \emph{critical parameter} $\beta_c$, which separates a \emph{disordered} or \emph{subcritical} phase $\beta<\beta_c$ from an \emph{ordered} or \emph{supercritical} one. While the model owes most of its physical relevance to the richness of its (near-)critical regime, the study of its off-critical behaviour has also been the subject of extensive work. A notable result is the proof of exponential decay of the two-point function for $\beta<\beta_c$ \cite{AizenmanBarskyFernandezSharpnessIsing1987} (see also \cite{DuminilTassionNewProofSharpness2016,duminil2019sharp} for alternative proofs), a property sometimes referred to as \emph{subcritical sharpness}. The supercritical regime is considerably more delicate to analyse: exponential decay of truncated correlations or \emph{supercritical sharpness} has recently been established by Duminil-Copin, Goswami, and Raoufi \cite{DGR20}, building on the supercritical sharpness result of Bodineau \cite{Bod05} for the closely related \emph{random cluster} (or \emph{Fortuin--Kasteleyn} (FK)) representation of the Ising model. 

Another geometric representation, central to several developments on the model (see e.g.~\cite{AizenmanGeometricAnalysis1982,AizenmanDuminilSidoraviciusContinuityIsing2015,AizenmanDuminilTriviality2021}), is the so-called \emph{random current representation}. First introduced by Griffiths, Hurst, and Sherman \cite{GriffithsHurstShermanConcavity1970} to prove concavity of the magnetisation in the presence of an external magnetic field, the full power of this representation emerged in the seminal work of Aizenman \cite{AizenmanGeometricAnalysis1982}. What makes it useful is a combinatorial identity called the \emph{switching lemma}, which provides probabilistic interpretations to product and differences of correlation functions of the Ising model. In particular, this identity offers a simple probabilistic interpretation of the truncated correlation of the model in terms of its \emph{double random current representation} (see Corollary \ref{coro:expression truncated switching}), suggesting that a detailed description of its supercritical behaviour could have many important consequences.

The purpose of this work is to study the supercritical regime of the percolation model induced by the double random current model. We prove a supercritical sharpness result for it, in the sense that \emph{local uniqueness} occurs with high probability, uniformly in the ``boundary conditions'' (which, in this context, are given by \emph{sources}). As a consequence, we obtain a simpler proof of the main result of \cite{DGR20}, and derive several statements of independent interest. 

Before stating the results of the paper, let us define the main protagonists here: the Ising model and its associated random current representation. 
We start by introducing some notation for subgraphs of $\mathbb{Z}^d$. Let $E(\mathbb Z^d):=\{xy: | x-y|_1=1\}$ where $|\cdot|_1$ denotes the $\ell^1$ norm on $\mathbb R^d$. We write $x\sim y$ when $xy\in E(\mathbb Z^d)$. If $\Lambda\subset \mathbb Z^d$, we define $E(\Lambda):=\{xy\in E(\mathbb Z^d): x,y\in \Lambda\}$, $\overline{E}(\Lambda):=\{xy\in E(\mathbb Z^d): x\in \Lambda \textup{ or }y\in \Lambda\}$, $\partial \Lambda:=\{x\in \Lambda: \exists y\notin \Lambda, \: x\sim y\}$, $\partial^{\textup{ext}}\Lambda:=\{y\in \mathbb Z^d\setminus \Lambda: \exists x\in \Lambda, \: x\sim y\}$ and $\overline{\Lambda}:=\Lambda\cup\partial^{\textup{ext}}\Lambda$.

\paragraph{The Ising model.} It will be convenient for us to define our models on \emph{weighted (multi-)graphs}. Let $G=(V,E,J)$ be a finite weighted graph with vertex set $V$, edge set $E$, and edge-weights $J:E\rightarrow \mathbb{R}^+$. 
Consider the Hamiltonian on $\Sigma_G\coloneqq \{-1,1\}^{V}$ given by
\begin{equation*}
	H_{G}(\sigma)\coloneqq - \sum_{xy\in \overline{E}} J_{xy}\sigma_x\sigma_y.
\end{equation*}
The Ising model on $G$ is defined as the probability measure $\nu_G$ on $\Sigma_G$ given by
\begin{equation*}
	\nu_G(\sigma) := \frac{1}{Z^{\textup{Ising}}_{G}} e^{-H_{G}(\sigma)},
\end{equation*}
where $Z^{\textup{Ising}}_{G}=\sum_{\sigma\in\Sigma_G} e^{- H_{G}(\sigma)}$ is the normalising \emph{partition function}. 
Expectations with respect to $\nu_G$ will be denoted by $\langle f(\sigma) \rangle_G$. 

For finite subsets $\Lambda\subset \Z^d$, we denote  by $\Lambda[\beta]$ the induced weighted graph with homogeneous weights equal to $\beta\geq 0$, i.e.~$\Lambda[\beta]=(\Lambda,E(\Lambda),\beta\mathbbm{1}_{E(\Lambda)})$. 
The measure $\nu_{\Lambda[\beta]}$ corresponds to the Ising model on $\Lambda$ with inverse temperature $\beta$ and \emph{free boundary conditions}.  
It is also convenient to consider the model with \emph{plus boundary condition}, which in this setting can be defined as follows. 
Let $\Lambda^\fg[\beta]$ denote the weighted (multi-)graph obtained from $\overline{\Lambda}[\beta]$ by collapsing all the vertices of $\partial^{\textup{ext}} \Lambda$ into a single \emph{ghost vertex} $\fg$. 
The Ising model on $\Lambda$ with inverse temperature $\beta$ and plus boundary conditions is then defined as 
\begin{equation*}
    \nu^+_{\Lambda[\beta]}:=\nu_{\Lambda^\fg[\beta]}(\ \cdot \mid \sigma_\fg=1).
\end{equation*}
It is classical \cite{GriffithsCorrelationsIsing1-1967,GriffithsCorrelationsIsing2-1967} that there exist corresponding infinite volume measures $\nu_\beta$ and $\nu_\beta^+$ respectively defined by
$$\nu_{\beta} \coloneqq \lim_{\Lambda \uparrow\Z^d} \nu_{\Lambda[\beta]},\qquad\nu^+_{\beta} \coloneqq \lim_{\Lambda \uparrow\Z^d} \nu^+_{\Lambda[\beta]}.$$
For $d\geq2$, the plus boundary condition \emph{breaks} the flip symmetry of the model for $\beta$ sufficiently large, thus leading to a phase transition. More precisely, the critical parameter defined as
$$\beta_c \coloneqq \inf\Big\{\beta\geq0 : m(\beta)\coloneqq \langle \sigma_0\rangle^+_\beta>0\Big\}$$
is \emph{non-trivial} (i.e.~$\beta_c\in(0,\infty)$) for every $d\geq2$, where $\langle\cdot\rangle_\beta^+$ denotes expectations with respect to $\nu_\beta^+$.

\paragraph{Random current representation.} Let $G=(V,E,J)$ be a finite weighted graph. Let $\Omega_{E}:= \mathbb{N}^{E}$ be the space of currents in $G$, where $\mathbb N=\lbrace 0,1,2,\ldots\rbrace$. Given $\n \in \Omega_{E}$, let 
$
\partial \n
:=
\lbrace x \in V : \Delta\n(x) \text{ is odd}\rbrace
$ be the set of sources of $\n$, where
$
\Delta\n(x):= \sum_{y\in V} \n_{xy}	
$
is the $\n$-degree of $x$.
Given $\n \in \Omega_E$, write 
\begin{equation*}
w_G(\n)=\prod_{xy\in E} \frac{J_{xy}^{\n_{xy}}}{\n_{xy}!}.
\end{equation*}
Given $A\subset V$ of even cardinality (we later say that $A$ is even), we define the probability measure $\mathbf{P}^A_{G}$ by the weights
\begin{equation*}
\mathbf{P}^A_{G}[\n]=\frac{\mathbbm{1}_{\partial \n=A} w_{G}(\n)}{Z^A_{G}},    
\end{equation*}
where $Z^A_{G}$ is a normalisation constant.
For two weighted graphs $G,H$ and source-sets $A,B$, we write $\mathbf{P}^{A,B}_{G,H}:=\mathbf{P}^{A}_{G}\otimes \mathbf{P}^{B}_{H}$ to denote the corresponding product measure (when $G=H$, we may omit $H$ from the subscript). We regard $\mathbf{P}^{A,B}_{G,H}$ as a percolation measure via the map $(\n,\m)\mapsto \widehat{\n+\m}$ defined by $(\widehat{\n+\m})_e:=\mathbbm{1}\{ \n_e+\m_e>0\}$. 
As observed in \cite{AizenmanDuminilSidoraviciusContinuityIsing2015}, when the graphs are subgraphs of $\mathbb Z^d$, it is possible to construct infinite volume limits of the above measures as $G,H\nearrow \mathbb Z^d$. In the setting with homogeneous edge-weight equal to $\beta\geq 0$, we denote the resulting double random current measure by $\mathbf{P}^{A,B}_{\beta}:=\lim_{\Lambda \uparrow \mathbb Z^d}\mathbf P^{A,B}_{\Lambda[\beta],\Lambda[\beta]}$.

Spin correlations of the Ising model can be expressed in terms of weighted sums over currents, see \eqref{eq: expression of correlation with currents}. One of the main interest of the random current representation lies in a simple but fundamental combinatorial identity called the switching lemma, see Lemma~\ref{lem:switching} below. It allows to express ratios and differences of correlation functions in terms of probability of connectivity events in the appropriate double random current measures. 
In particular, one can show that, for every $x,y\in \mathbb Z^d$,
\begin{equation*}
    \langle \sigma_x\sigma_y\rangle_{\beta}^2=\mathbf{P}^{\emptyset,\emptyset}_{\beta}[0\connect{\widehat{\n_1+\n_2}\:}x].
\end{equation*}
Combined with the tools developed in \cite{AizenmanDuminilSidoraviciusContinuityIsing2015}, this suffices to prove that the percolation phase transition of $\mathbf{P}^{\emptyset,\emptyset}_{\beta}$ occurs at $\beta_c$. 

Despite these properties, the double random current percolation measure remains difficult to analyse. Indeed, the measures $\mathbf P^{A,B}_{G,H}$ introduced above do not satisfy the FKG inequality, a fact that implies, in particular, that several monotonicity properties (in volume, edge-weight, etc.) are unavailable.

For $L\geq 1$, we let $\Lambda_L=[-L,\ldots,L]^d\cap \mathbb{Z}^d$. Our main result is a supercritical sharpness result for double random currents. Below, for $L\geq 1$, we let $U(L)$ denote the \emph{local uniqueness event at scale $L$}, that is, the event that there exists at least one cluster crossing $\Lambda_{8L}\setminus \Lambda_L$, and that any two clusters crossing $\Lambda_{4L}\setminus\Lambda_{2L}$ are connected in $\Lambda_{8L}\setminus \Lambda_L$.

\begin{Thm}[Supercritical sharpness for random currents]\label{thm:unique_currents}
Let $d\geq 3$ and $\beta>\beta_c$. Then,  \begin{equation}\label{eq:unique_currents}
        \lim_{L\to\infty} \inf_{A,B\subset \partial\Lambda_{10L}} \mathbf{P}^{A,B}_{\Lambda_{10L}[\beta], \Lambda_{10L}[\beta]} [U(L)] = 1,
    \end{equation}
    where the infimum is taken over all subsets of even cardinality.
\end{Thm}

\begin{Rem}\label{rem:2D local uniqueness}
    It is folklore that \eqref{eq:unique_currents} also holds for $d=2$. In fact, a stronger result, involving only the single current measure $\mathbf P^A_{\Lambda_{10L}[\beta]}$ (with $A\subset \partial \Lambda_{10L}$ of even cardinality), holds.  
    We only sketch the argument here in the simpler case $A=\emptyset$, which follows an observation of \cite{hansen2023uniform}. By planarity, local uniqueness follows from the existence of a circuit that surrounds a dyadic annulus with probability tending to $1$. Since the single sourceless random current measure stochastically dominates the high-temperature expansion (see e.g.~\cite{AizenmanDuminilTassionWarzelEmergentPlanarity2019}), it suffices to prove the existence of a circuit for the latter. The high-temperature expansion at parameter $\beta$ is distributed as the interfaces separating the sign clusters of a wired Ising model on the dual lattice at parameter $\beta^*$. By the Kramers--Wannier duality, if $\beta>\beta_c$, then $\beta^*<\beta_c$. For the subcritical Ising model, the probability of a sign-cluster crossing the annulus decays exponentially, uniformly over the boundary conditions \cite{Higuchi93}. When such a crossing does not happen, we get the desired existence of a circuit. The general case, with source set $A$ arbitrary on the boundary, follows by similar considerations\footnote{Adding sources to the single random current measure induces spin boundary conditions when passing to the dual Ising model: letting $A=\{a_1,\ldots,a_{2k}\}$ be the sources listed in cyclic order, they induce alternating sign boundary conditions on the arcs connecting $a_i$ to $a_{i+1}$. Because the exponential decay of a sign-cluster crossing is uniform in the boundary conditions, this does not affect the previous argument.}. 
\end{Rem}

When combined with a standard coarse-graining argument, Theorem~\ref{thm:unique_currents} implies that, for every $\beta>\beta_c$ and every $N$ large enough, the percolation model $\mathbf{P}^{S,T}_{\Lambda_N[\beta],\Lambda_N[\beta]}$ behaves essentially like a highly supercritical Bernoulli percolation ``away'' from the sources $S$ and $T$. Indeed, fix $L$ sufficiently large (to be chosen later), and consider the renormalised site percolation process $\eta^L$ on $\Lambda_{\tilde{N}}$, with $\tilde{N}:=\lfloor N/L \rfloor-10$, defined by
\begin{equation*}
	\eta^L_z := \mathbbm{1}{\{\eta(Lx+\cdot)\in U(L)\}}, ~~~~~~ z\in  \Lambda_{\tilde{N}},
\end{equation*}
where $\eta=\widehat{\n+\m}$ is distributed as the percolation marginal of $\mathbf{P}^{S,T}_{\Lambda_N[\beta],\Lambda_N[\beta]}$. 
First, by the definition of the event $U(L)$, every cluster of $\eta^L$ deterministically induces a cluster of $\eta$.
Second, by the domain Markov property, one has 
\begin{equation*}
	\mathbf{P}^{S,T}_{\Lambda_N[\beta],\Lambda_N[\beta]}[\eta^L_{z}=1 \,|\, ( \eta^L_{w}, \, w\in \Lambda_{\tilde{N}} \setminus \Lambda_{20}(z) )]\geq  \inf_{A,B\subset \partial \Lambda_{10L}}\mathbf{P}^{A,B}_{\Lambda_{10L}[\beta],\Lambda_{10L}[\beta]}[U(L)],
\end{equation*}
for every $x\in \Lambda_{\tilde{N}}\setminus (S_L\cup T_L)$, where $S_L:= \{z\in \Lambda_{\tilde{N}} :\, \Lambda_{10L}(Lz)\cap S\neq\emptyset \}$ (and $T_L$ defined analogously). 
It then follows from Theorem~\ref{thm:unique_currents} and the main result of \cite{liggett1997domination} that $\eta^L$ stochastically dominates a Bernoulli site percolation on $\Lambda_{\tilde{N}}\setminus (S_L\cup T_L)$ whose parameter can be made arbitrarily close to $1$, provided that $L$ is chosen large enough.
Consequently, many properties that are known to hold for strongly supercritical Bernoulli percolation automatically transfer to $\mathbf{P}^{S,T}_{\Lambda_N[\beta],\Lambda_N[\beta]}$. For example, one can deduce that the convergence in \eqref{eq:unique_currents} is exponential, that is,
\begin{equation*}
\inf_{A,B\subset \partial\Lambda_{10L}} \mathbf{P}^{A,B}_{\Lambda_{10L}[\beta], \Lambda_{10L}[\beta]} [U(L)]\geq1-e^{-cL}.
\end{equation*}
Combined with the switching lemma, this implies several mixing results for the Ising model, which we state below.

\paragraph{Consequences.} 
The following are easy consequences of Theorem~\ref{thm:unique_currents}. We state these corollaries for $d\geq3$, but the proofs only use the fact that \eqref{eq:unique_currents} holds, and therefore also work for the planar case $d=2$ (see Remark~\ref{rem:2D local uniqueness}).
Below, for an Ising Gibbs measure $\langle\cdot\rangle$ and finite subsets $A,B\subset \mathbb Z^d$, we let $\sigma_A:=\prod_{x\in A}\sigma_x$ and
\begin{equation*}
    \langle \sigma_A;\sigma_B\rangle:=\langle \sigma_A\sigma_B\rangle -
    \langle \sigma_A\rangle \langle \sigma_B\rangle.
\end{equation*}

\begin{Coro}[Exponential decay of truncated correlations]\label{coro: truncated} 
For every $d\geq 3$ and $\beta>\beta_c$, there exists $c=c(d,\beta)>0$ such that the following holds. For every $n\geq1$ and every finite subsets $A\subset \Lambda_n$ and $B\subset \Z^d\setminus\Lambda_{2n}$ 
    \begin{equation}\label{eq:truncated}
        \langle \sigma_A ; \sigma_B \rangle_\beta^+ \leq 
        e^{-cn} \langle \sigma_A\rangle_{\beta}^+\langle \sigma_B\rangle_{\beta}^+.
    \end{equation}
In particular, 
for every $x\in\Z^d$,
    \begin{equation}\label{eq:truncated_pt_to_pt}
        \langle \sigma_0 ; \sigma_x \rangle_\beta^+ \leq 
        e^{-c|x|}.
    \end{equation} 
\end{Coro}

\begin{Rem}\label{rem: free even mixing}
    Our proof of \eqref{eq:truncated} also implies that the same inequality holds for the free measure $\langle\cdot\rangle_\beta$ if  $A$ and $B$ are assumed to have even cardinality. 
\end{Rem}

\begin{Coro}[Exponential mixing for Ising]\label{cor: ising mixing}
	Let $d\geq 3$ and $\beta>\beta_c$. There exists $c=c(d,\beta)>0$ such that, for every $n\geq1$ and all events $E$ and $F$ depending on $\Lambda_n$ and $\mathbb Z^d\setminus \Lambda_{2n}$, respectively, 
	\begin{equation*}
		\Big| \nu^+_{\beta}[E\cap F]-\nu^+_{\beta}[E]\nu^+_{\beta}[F] \Big| \leq e^{-cn}.
	\end{equation*}
\end{Coro}

\begin{Rem}
The stronger \emph{exponential ratio mixing} statement (i.e.~$|\nu^+_{\beta}[E\cap F]-\nu^+_{\beta}[E]\nu^+_{\beta}[F]|\leq e^{-cn}\nu^+_{\beta}[E]\nu^+_{\beta}[F]$), as derived for other measures in Corollary \ref{cor: perco ratio mixing} below, is not true here due to the fact that there are several infinite volume Gibbs measures. To see this, consider the events $E=\{\sigma_0=1\}$ and $F=\{\sigma_x=-1 \, \forall x\in \partial^{\textup{ext}} \Lambda_{2n}\}$ and observe that, by the domain Markov property, $\nu^+[E\: |\: F]/\nu^+[E]=\nu_{\Lambda_{2n}[\beta]}^-[E]/\nu^+[E]$ (where $\nu^-_{\Lambda[\beta]}=\nu_{\Lambda^\fg[\beta]}(\ \cdot \mid \sigma_\fg=-1)$), which does not converge to $1$ since $\nu^+_\beta\neq \nu^-_\beta$ when $\beta>\beta_c$. Note however that by summing \eqref{eq:truncated}, ratio mixing holds for random variables of the form $X=\sum_S a_S \sigma_S$, with $a_S\geq 0$. In particular, this includes\footnote{Indeed, observe that $\mathds{1}_{\sigma_x=1 \: \forall x\in A}=2^{-|A|}\prod_{x\in A}(\sigma_x+1)$.} events of the form $\{\sigma_x=1 \,\, \forall x\in A\}$.
\end{Rem}

Our last corollary concerns the ratio mixing of two percolation representations of the Ising model. 
Let $\mathbf{P}_\beta$ be the (unique) infinite-volume random current measure, given by the weak limit of $\mathbf{P}^S_{\Lambda[\beta]}$ as $\Lambda\uparrow\Z^d$, for arbitrary (even) boundary sources $S\subset \partial\Lambda$. The existence and uniqueness of this measure follow from \cite{AizenmanDuminilSidoraviciusContinuityIsing2015} and \cite{BodineauTranslationGibbsIsing2006} (see also \cite{RaoufiGibbsMeasures}), respectively. 
Analogously, one can also define a (unique) infinite volume random cluster (or FK) measure $\phi_\beta$ as the weak limit of the finite-volume measures $\phi^\xi_{\Lambda[\beta]}$ as $\Lambda\uparrow \Z^d$, for arbitrary boundary partitions $\xi\in \textup{Part}(\partial \Lambda)$ (the set of partitions of $\partial \Lambda$).  We refer the reader to Section~\ref{sec:preliminaries} for more details (including the definition of $\phi^\xi_{\Lambda[\beta]}$).

\begin{Coro}[Exponential ratio mixing for the percolation representations]\label{cor: perco ratio mixing} 
Let $d\geq 3$ and $\beta>\beta_c$. There exists $c=c(d,\beta)>0$ such that the following properties hold. For every $n\geq 1$ and every events $E$ and $F$ depending on edges in $\Lambda_n$ and $\mathbb Z^d\setminus \Lambda_{2n}$, respectively,
	\begin{equation}\label{eq:ratio-mixing-currents}
		\Big| \mathbf{P}_{\beta}[E\cap F]-\mathbf{P}_{\beta}[E]\mathbf{P}[F] \Big| \leq e^{-cn}\mathbf{P}_{\beta}[E]\mathbf{P}_{\beta}[F],
	\end{equation}
    and
    	\begin{equation}\label{eq:ratio-mixing-FK}
		\Big| \phi_{\beta}[E\cap F]-\phi_{\beta}[E]\phi_{\beta}[F] \Big| \leq e^{-cn}\phi_{\beta}[E]\phi_{\beta}[F].
	\end{equation}
\end{Coro}

\begin{Rem} The above ratio mixing property also holds for the so-called \emph{high-temperature expansion} of the Ising model (also known as loop $O(1)$ model), since it has the same distribution as the set of edges where a current sampled from $\mathbf{P}_{\beta}$ takes an odd value  (see e.g.~\cite{AizenmanDuminilTassionWarzelEmergentPlanarity2019}). 
 Mixing for this representation was recently derived in \cite{hansen2026supercritical} using different methods, but relying on the results of \cite{DGR20}. 
\end{Rem}

\begin{Rem}
    By the domain Markov property of the respective measures, Corollary~\ref{cor: perco ratio mixing} readily implies that if $f_n^{S}$ (resp. $g_n^\xi$) denotes the Randon--Nikodym derivative between the measures $\mathbf{P}^S_{\Lambda_{2n}[\beta]}$ and $\mathbf{P}_\beta$ (resp. $\phi^\xi_{\Lambda_{2n}[\beta]}$ and $\phi_\beta$) restricted to $\Lambda_n$, then $\sup_{S\subset \partial\Lambda_{2n}} \|1-f^S_n\|_{\infty} \leq e^{-cn}$ (resp. $\sup_{\xi\in \mathrm{Part}({\partial\Lambda_{2n}})} \|1-g^\xi_n\|_{\infty} \leq e^{-cn}$). In particular, this also implies a corresponding exponential upper bound on their total variation distance.
\end{Rem}

\begin{Rem}
    Let us note that appropriate finite-volume versions of Corollaries~\ref{coro: truncated}, \ref{cor: ising mixing},  \ref{cor: perco ratio mixing} follow readily from our proofs, see Section~\ref{sec:corollaries}.
\end{Rem}

\noindent\textbf{Description of the proof.} Let us now describe the main novelties of our approach. To this end, we first recall the strategy of \cite{DGR20}. 

The starting point in \cite{DGR20} is the observation that \eqref{eq:truncated_pt_to_pt} follows easily from exponential mixing for the FK-Ising model (see \eqref{eq:ratio-mixing-FK}), together with the supercritical sharpness result established in \cite{Bod05}. Then, they show that \eqref{eq:ratio-mixing-FK} follows\footnote{In fact, their proof yields a non-ratio mixing statement (see \cite[Theorem 1.3]{DGR20}), from which ratio mixing follows (see \cite[Corollary 1.4]{DGR20}) via an abstract theorem from \cite{Ale98}. Here, we obtain ratio mixing directly.} from a coarse-graining argument once a super-polynomial bound on the difference between the edge densities of the free and wired FK-Ising measures is established (see \cite[Proposition 1.5]{DGR20}). 
Most of the work in \cite{DGR20} is therefore devoted to proving \cite[Proposition 1.5]{DGR20}.
Thanks to the switching lemma (see Lemma~\ref{lem:switching} below) one can express the difference between the marginals of the wired and free FK-Ising measures at an edge $e=\{x,y\}$ in terms of random currents. Proposition 1.5 in \cite{DGR20} follows from showing that, in a double random current with sources $A=\{x,y\}$ and $B=\emptyset$, the probability that $x$ and $y$ are both connected to the boundary of a ball of radius $n$, but not to each other, decays super-polynomially fast. This is achieved in two steps. The first consists in establishing a surface-order exponential upper bound on the probability that two large boxes are not connected in the double random current, see \cite[Proposition 4.1]{DGR20}. Such a bound is known for the random cluster model by supercritical sharpness, and is transferred to double random currents in \cite{DGR20} through a box-merging trick. The second step relies on a highly technical multi-valued map argument to show that, if large boxes are connected to each other with very high probability (as established in the first step), then two long paths starting from $x$ and $y$ must also be connected with very high probability.

Our overall strategy differs substantially from that of \cite{DGR20}, beginning with the fact that our main result establishes supercritical sharpness for double random currents, a result not obtained in \cite{DGR20}. Once this is established, all the consequences discussed above (including the results of \cite{DGR20}, as well as some new ones) follow rather easily. Nevertheless, our proof of Theorem~\ref{thm:unique_currents} shares some similarities with the proof of \cite[Proposition 1.5]{DGR20} outlined above. Indeed, our argument also proceeds in two steps. The first consists in establishing a surface-order exponential bound on disconnection probabilities. The proof is essentially the same as that of \cite[Proposition 4.1]{DGR20}. We note, however, that the proof in \cite{DGR20} only applies to double random currents in which at least one of the currents is sourceless (i.e.~$B=\emptyset$). Since we need to handle arbitrary pairs of boundary sources, we make use of a monotonicity argument from \cite{duminil2021conformal}.

The second step, on the other hand, is very different from that of \cite{DGR20}. Instead of relying on tedious multi-valued map argument, we build on an observation of \cite{SeveroSlabIsing2023}, introduced to give a short proof of supercritical sharpness for the FK-Ising model. The idea, inspired by \cite{benjaminitassion}, is that surface-order exponential bounds imply local uniqueness in the presence of ``sprinkling'', that is, when one is allowed to open additional edges independently with small probability (see Proposition~\ref{prop:uniqueness with sprinkling}). For percolation models that are monotone in the parameter $\beta$ (such as the FK-Ising model), the sprinkled model can be compared to the original model at a slightly larger parameter $\beta'>\beta$. However, as discussed above, the random current model lacks such monotonicity properties. To overcome this difficulty, we show in Proposition~\ref{prop:coupling} that the random current model can be represented as a sprinkled random current model on a \emph{modified weighted graph}. The argument from the first step readily extends to this modified graph (see Proposition~\ref{prop:uniform connectivity}), and the desired result then follows by combining these ingredients.

\hfill

\noindent \textbf{Acknowledgements.} CP was supported by an EPSRC New Investigator Award (UKRI1019). RP acknowledges support from the Swiss National Science Foundation and from the ERC grant
Vortex (No.\ 101043450).

\hfill

\noindent\textbf{Statement of AI use.} We acknowledge the use of LLMs for grammatical corrections and language editing. They were not used to generate ideas or proofs, or to suggest references.

\section{Preliminaries}\label{sec:preliminaries}

In this section, we recall classical results regarding the Ising model and its random cluster representation. Below, when working with plus boundary conditions on a weighted graph $G=(V,E,J)$, we additionally assume that $V\subset \mathbb Z^d$.

\subsection{Correlation inequalities}

Recall that if $A\subset \mathbb Z^d$ is finite,  $\sigma_A=\prod_{x\in A} \sigma_x$.

\begin{Prop}[\hspace{1pt}{\cite{GriffithsCorrelationsIsing1-1967}}]\label{prop: Griffiths} Let $G=(V,E,J)$ be a weighted graph and $\#$ be either free or plus boundary conditions. For every finite $A,B\subset V$, one has
\begin{equation*}
	\langle \sigma_A\rangle_{G}^\#\geq 0, \qquad \langle \sigma_A;\sigma_B\rangle_{G}^\#
\geq 0.
\end{equation*}
\end{Prop}

\begin{Prop}[\hspace{1pt}{\cite{GriffithsCorrelationsIsing1-1967,kelly1968general}}]\label{prop:consequencegriffiths} Let $G=(V,E,J)$ and $H=(V',E',J')$ be two weighted graphs such that $G\subset H$ (in the sense that $V\subset V', E\subset E'$, and $J\leq J'$). For every finite $A\subset V$, one has
\begin{equation}\label{eq:monot Griffiths}
	\langle \sigma_A\rangle_{G}\leq \langle \sigma_A\rangle_{H}.
\end{equation}
\end{Prop}

We now recall the FKG inequality for the Ising model.

\begin{Prop}\label{prop:FKG}
Let $G=(V,E,J)$ be a weighted graph and $\#$ be either free or plus boundary conditions. Let $f,g:\{-1,1\}^V\rightarrow \mathbb R$ be increasing and bounded functions. Then,
\begin{equation*}
    \langle f; g\rangle_G^\#:=\langle fg\rangle_G^\#-\langle f\rangle^\#_G\langle g\rangle^\#_G\geq 0.
\end{equation*}
\end{Prop}

Finally, we will need an inequality due to Lebowitz \cite{Lebowitz1977CoexistencePhasesIsing} (see e.g.~\cite{gunaratnam2026supercritical} for a modern proof). To state it, observe that we can classically enhance the definition of the Ising model from above by adding an external magnetic field term: if $G=(V,E,J)$ is a finite weighted graph and $\mathsf{h}\in \mathbb R^V$ is a (non-homogeneous) external magnetic field, we may define $H_{G,\mathsf{h}}(\sigma):=H_G(\sigma)+\sum_{x\in V}\mathsf{h}_x\sigma_x$. The resulting Gibbs measure is denote $\nu_{G,\mathsf{h}}$ (or $\langle \cdot \rangle_{G,\mathsf{h}}$).

\begin{Prop} \label{prop: lebowitz} 
Fix a finite weighted graph $G = (V,E,J)$. Consider two external magnetic fields $\mathsf{h}$ and $\mathsf{h}'$ satisfying $|\mathsf{h}'_x|\le \mathsf{h}_x$ for every $x\in V$. For every $A, B\subset V$, one has
\begin{equation*}
\langle \sigma_A \sigma_B \rangle - \langle \sigma_A \sigma_B \rangle' \ge |\langle \sigma_A \rangle \langle\sigma_B \rangle' - \langle \sigma_A \rangle' \langle \sigma_B \rangle|,
\end{equation*}
where $\langle \cdot \rangle$ and $\langle \cdot \rangle'$ stand for the expectations under the measures $\nu_{G,\mathsf{h}}$ and $\nu_{G,\mathsf{h}'}$ respectively.
\end{Prop}
\begin{Rem}\label{rem:twist of lebowitz} Fix $x\in V$ and consider an external magnetic field $h$ such that $h_y=0$ for all $y\in V\setminus \{x\}$. If we take the limit $\mathsf{h}_x\rightarrow \infty$, then $\nu_{G,\mathsf{h}}\rightarrow \nu_{G}(\ \cdot \mid \sigma_x=1)$. In particular, the above result holds with $\langle \cdot\rangle=\langle\ \cdot \mid  \sigma_x=1\rangle_{G}$ and $\langle \cdot \rangle'=\langle \cdot\rangle_{G}$.
\end{Rem}

\subsection{Random cluster model}

Let $G=(V,E,J)$ be a finite weighted (multi-)graph. The \emph{free} FK-Ising measure on $G$ is given, for $\omega\in \{0,1\}^E$, by 
\[
\phi_{G}[\omega]=\frac{1}{Z^{\rm FK}_G}\prod_{e\in E} (1-e^{-2J_e})^{\omega_e} (e^{-2J_e})^{1-\omega_e} \; 2^{k(\omega)},
\]
where $Z^{\rm FK}_G$ is a normalisation constant and $k(\omega)$ denotes the number of connected components of the graph with vertex set $V$ and edge set $\{e\in E: \omega_e=1\}$.

Given $H$ a finite weighted subgraph of $\mathbb Z^d$ and $\xi$ a partition of $\partial V$, we write $H^{\xi}$ to denote the multigraph obtained from $H$ after identifying vertices in the same partition class of $\xi$ to a single vertex. 
When $G=H^{\xi}$, we will simply write $\phi^{\xi}_{H}$.

\paragraph{Edwards-Sokal coupling.} 

We now recall how to couple the Ising model with its random cluster measure.

\vspace{3pt}

\noindent\textit{From percolation to spin model.} Given $\omega\sim \phi_{G}$, we can sample a spin configuration $\sigma\sim \langle \cdot \rangle_{G}$ as follows.
\begin{enumerate}
\item[(i)] Consider a sequence of independent random variables $\sigma_{\mathcal{C}}\in \{\pm 1\}$, indexed by the connected components $\mathcal{C}$ of $\omega$, such that $\mathbb{P}[\sigma_{\mathcal{C}}=1]=1/2$.
\item[(ii)] For every cluster $\mathcal{C}$ and $x\in \mathcal{C}$, set $\sigma_x =\sigma_{\mathcal{C}}$.
\end{enumerate}

\noindent \textit{From spin model to percolation.} Conversely, given a spin configuration $\sigma\sim \langle \cdot \rangle_{G}$, we can sample a percolation configuration $\omega\sim \phi_G$ as follows.

\begin{enumerate}
\item[(i)] For each edge $xy\in E$ such that $\sigma_x\neq \sigma_y$, set $\omega_{xy}=0$. 
\item[(ii)] For each edge $xy\in E$ such that $\sigma_x=\sigma_y$, let $\omega_{xy}=1$ with probability equal to $1-e^{-2J_e}$, independently of the other edges.
\end{enumerate}
The above procedure allows to relate the two-point function of the Ising model to that of its associated random cluster representation.

\begin{Coro}\label{cor: ES correlations}
Let $G=(V,E,J)$ be a finite weighted graph. For every $x,y\in V$, 
\begin{align}\label{eq: es cor: 3}
\langle \sigma_x \sigma_y \rangle_{G}
&=
\phi_{G}[x\longleftrightarrow y]. 
\end{align}
More generally, given any $A \subset V$ of even cardinality,
\begin{equation} \label{eq: edwards-sokal-genA}
    \langle \sigma_A \rangle_G = \phi_G[\mathcal F_A],
\end{equation}
where $\mathcal F_A$ is the event that every connected component of $\omega$ intersects $A$ an even number of times.

\end{Coro}

Recall the definition of $U(n)$ from above Theorem \ref{thm:unique_currents}. The following is the analogue of Theorem~\ref{thm:unique_currents}, but for the FK-Ising model, which was proved by Bodineau~\cite{Bod05} (see also \cite{SeveroSlabIsing2023}). 
\begin{Prop}[Supercitical sharpness for FK-Ising] \label{prop: supercritical sharpness FK-Ising}
Let $d\geq 2$ and $\beta>\beta_c$. Then, 
\begin{equation*}
\lim_{n\to \infty} \inf_{\xi}\phi^{\xi}_{\Lambda_{10n}[\beta]}[U(n)]=1,
\end{equation*}
where the infimum is taken over all partitions of $\partial\Lambda_{10n}$.
\end{Prop}

\begin{Rem}\label{rem: connectivity fk} Using classical renormalisation arguments (as described below Theorem \ref{thm:unique_currents}), one can use the above local uniqueness result to prove the following estimate, which we will use in the paper: for every $d\geq 2$ and $\beta>\beta_c$, there exists $c>0$ such that for every $N\geq n\geq0$, and every $x,y\in \Lambda_N$, one has
\begin{equation*}
    \phi_{\Lambda_N[\beta]}[\Lambda_{n}(x)\leftrightarrow \Lambda_{n}(y)]\geq 1-e^{-cn^{d-1}}.
\end{equation*}
\end{Rem}

\subsection{Switching lemma}

The following results are classical properties of random currents. We refer to \cite{DuminilLecturesOnIsingandPottsModels2019,Panis2024applications} for proofs.

Fix $G=(V,E,J)$ a finite weighted graph. By  Taylor expanding the exponential, one can show that $\sum_\sigma \sigma_A e^{-H_{G}(\sigma)} = 2^{|V|} Z^A_G$, where we recall that $Z^A_G:=\sum_{\partial\n=A} w_G(\n)$. In particular,
\begin{equation}\label{eq: expression of correlation with currents}
    \langle \sigma_A\rangle_{G}=\frac{Z^A_G}{Z^\emptyset_G}.
\end{equation}
The sums on the numerator and denominator range over distinct subsets of $\Omega_E$. As a result, it is not clear how to give a probabilistic interpretation of the above ratio. This problem can be solved using the following fundamental combinatorial identity.

\begin{Lem}[Switching lemma]\label{lem:switching} Let $G=(V,E,J)$ be a finite weighted graph. For every $A,B\subset V$ even, one has
   \begin{equation*}
      Z^A_G Z^B_G = \sum_{\substack{\partial \n=A\Delta B \\ \partial\m=\emptyset}} w_G(\n)w_G(\m) \mathbbm{1}\{\widehat{\n+\m}\in \mathcal{F}_A\},
   \end{equation*}
   where $\mathcal{F}_A$ denotes the event that $|C\cap A|$ is even for every cluster $C$ of $\widehat{\n+\m}$, and $A\Delta B$ denotes the symmetric difference of the sets $A$ and $B$. 
\end{Lem}

Let $\Lambda\subset \mathbb Z^d$ and $\beta\geq 0$. We can also represent correlation functions under the measure $\langle \sigma\rangle_{\Lambda[\beta]}^+$ in terms of currents on $\Lambda^\fg[\beta]$. Letting $\Omega^\fg_{E(\Lambda)}$ denote the set of currents on $\Lambda^\fg[\beta]$, one has, for every set $A\subset V\cup \{\fg\}$ even
\begin{equation*}
    \langle\sigma_A\rangle_{\Lambda[\beta]}^+=\frac{\sum_{\n \in \Omega_{E(\Lambda)}^\fg: \:\partial \n=A}w_{\Lambda^\fg[\beta]}(\n)}{\sum_{\n \in \Omega_{E(\Lambda)}^\fg: \:\partial \n=\emptyset}w_{\Lambda^\fg[\beta]}(\n)}.
\end{equation*}
The following result is a direct consequence of Lemma \ref{lem:switching}.
\begin{Coro}\label{coro:expression truncated switching} Let $d\geq 2$ and $\beta\geq 0$. For every $\Lambda\subset \mathbb Z^d$ finite and every finite $A,B\subset \Lambda$,
\begin{equation*}
    \langle \sigma_A;\sigma_B\rangle_{\Lambda[\beta]}^+=\langle \sigma_A\sigma_B\rangle_{\Lambda[\beta]}^+\mathbf{P}^{A\Delta B,\emptyset}_{\Lambda^\fg[\beta],\Lambda^\fg[\beta]}[\mathcal{F}_A^c].
\end{equation*}
\end{Coro}

\section{Proof of Theorem~\ref{thm:unique_currents}}

The proof of Theorem~\ref{thm:unique_currents} will follow rather easily from three key propositions, which we state below. The first one allows us to represent the random current measure on $\Lambda\subset \mathbb{Z}^d$ as a Bernoulli sprinkling of another random current measure on a slightly different (weighted) graph. This is the key novel property that ultimately allows us to avoid the use of the technical multi-valued map principle that appears in \cite{DGR20}, and the proof is very simple. 
The second proposition shows that disconnection probabilities for the double random currents on this modified graph satisfy a surface-order exponential decay. The proof is an adaptation of \cite[Proposition 4.1]{DGR20}, combined with a simple monotonicity statement which allows us to treat arbitrary boundary conditions, i.e.\ sources. This proposition implies, by a simple union bound, that the boundary clusters in $\Lambda_L$ are $C_p(\log L)^{1/(d-1)}$-dense. 
The last proposition, which is imported from \cite{SeveroSlabIsing2023}, says that any Bernoulli sprinkling of such a dense cloud of clusters leads to the local uniqueness event $U(L)$.

For any $\Lambda \subset\mathbb{Z}^d$ finite, we define $\tilde{\Lambda}$ to be the graph obtained by duplicating each edge $xy\in E(\Lambda)$, adding a ``midpoint'' $z=z(xy)$ to one of the copies, and denoting $xy$, $xz$ and $yz$ the edges of $\tilde{\Lambda}$ then defined. Given $0<\beta'<\beta$, we also define the weighted graph $\tilde{\Lambda}[\beta,\beta']$ given by $J_{xy}=\beta'$ and $J_{xz}=J_{yz}=\xi$, where $\xi=\xi(\beta-\beta')$ is such that 
$$\tanh(\xi)^2=\tanh(\beta-\beta').$$ 
\begin{figure}[htb]
    \centering
    \includegraphics[width=0.35\linewidth]{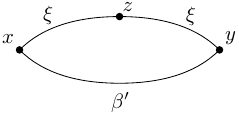}
    \caption{An illustration of (part of) the weighted graph $\tilde{\Lambda}[\beta,\beta']$.}
    \label{fig:doubleedge}
\end{figure}
Recall that for any current $\n$ in $\Lambda$, we define the percolation configuration $\hat{\n}$ in $\Lambda$ by setting $\hat{\n}_{xy}= \mathbbm{1}\{\n(xy)>0\}$. Similarly, for any current $\m$ in $\tilde{\Lambda}$, we can naturally associate the percolation configuration $\hat{\m}$ in  $\Lambda$ by setting $\hat{\m}_{xy}= \mathbbm{1}\{\m(xy)>0 \text{ or } \m(xz)\m(yz)>0\}$.  
The following proposition establishes that $\mathbf{P}^{S}_{\Lambda[\beta]}$ is a sprinkling of $\mathbf{P}^{S}_{ \tilde{\Lambda}[\beta,\beta']}$. This is the main new technical input of our proof, which ultimately allows us to apply the ``uniqueness with sprinkling'' approach of \cite{SeveroSlabIsing2023} to (double) random currents despite its lack of monotonicity in $\beta$. 
The proof is presented in Section~\ref{sec:coupling}.

\begin{Prop}\label{prop:coupling}
Let $0\leq \beta'<\beta$. There exists $\varepsilon=\varepsilon(\beta,\beta')>0$ such that the following holds. For every finite subsets $S\subset \Lambda\subset \mathbb{Z}^d$, if $\n\sim \mathbf{P}^{S}_{\Lambda[\beta]}$ and $\m\sim\mathbf{P}^{S}_{ \tilde{\Lambda}[\beta,\beta']}$ then $\hat{\n}$ is distributed as the union of $\hat{\m}$ with an independent \textup{Ber}$(\varepsilon)$ percolation on $\Lambda$.
\end{Prop}

The next proposition is similar to \cite[Proposition 4.1]{DGR20} and establishes a surface-order exponential bound for the probability that a box is connected to far away. The crucial differences are the following: (1) here we work on the modified graph $\widetilde{\Lambda_N}[\beta,\beta']$, (2) our statement involves box-to-boundary events instead of box-to-box, (3) we allow for arbitrary pairs of boundary sources. We note that (2) is precisely what allows us to harvest a monotonicity in boundary sources to obtain (3). The proof is presented in Section~\ref{sec:surface connection}.

\begin{Prop}
    \label{prop:uniform connectivity}
 	For every $\beta'>\beta_c$, there exists $c=c(\beta')>0$ such that, for every $\beta>\beta'$, every $N \geq 2n$, every $x \in \Lambda_{N-n}$, and every even $A,B\subset \partial \Lambda_{N}$, we have
 	$$\mathbf{P}^{A, B}_{\widetilde{\Lambda_N}[\beta,\beta'], 
    \widetilde{\Lambda_N}[\beta,\beta']}\Big[\lr{}{\widehat{\n_1 + \n_2}\:}{\Lambda_n(x)}{\partial \Lambda_{N-2n}}\Big] \geq 1 - e^{-cn^{d-1}}\,.$$
\end{Prop}

The final ingredient of our proof is Proposition~\ref{prop:uniqueness with sprinkling} below. In words, it says that if a configuration has $\ell_L$-dense boundary-clusters in $\Lambda_{8L}$ for $\ell_L=o(\log L)$ (in the sense that they touch every box of size $\ell_L$), then local uniqueness holds with high probability if we independently add edges with arbitrarily small probability. Such a result was first obtained by Benjamini and Tassion \cite{benjaminitassion} for $\ell_L = 0$, and later used in the present form by Severo \cite{SeveroSlabIsing2023} to prove local uniqueness for the FK-Ising model. In fact, Proposition~\ref{prop:uniqueness with sprinkling} is a small modification of \cite[Proposition 2.2]{SeveroSlabIsing2023}. We explain in Section~\ref{sec:proof uniqueness with sprinkling} how to deduce the former for the latter. 

Given $0\leq\ell\leq L$, consider the set of bond percolation configurations
\begin{equation*}
\mathcal{A}_L=\mathcal{A}_L(\ell)=\{\omega \subset E(\Lambda_{8L}): \forall\, x\in \ell\mathbb{Z}^d \cap \Lambda_{8L}, \Lambda_{\ell}(x) \text{ is connected to } \partial\Lambda_{8L} \text{ in } \omega\}. 
\end{equation*}
Below, $(\mathbb{P}_{\varepsilon},\gamma)$ denotes a Bernoulli bond percolation on $\mathbb{Z}^d$ of parameter $\varepsilon$. 

\begin{Prop}\label{prop:uniqueness with sprinkling}
Let $d\geq 2$ and let $(\ell_L)_{L\geq 1}$ be a sequence such that $\ell_L=o(\log L)$ as $L\to\infty$. Then, for every $\varepsilon>0$, one has
\begin{equation*}
\lim_{L\to \infty} \inf_{\omega \in \mathcal{A}_L(\ell_L)}\mathbb{P}_{\varepsilon}[\omega\cup \gamma \in U(L)]=1.    
\end{equation*}
\end{Prop}

We are now ready to prove Theorem~\ref{thm:unique_currents}.

\begin{proof}[Proof of Theorem~\textup{\ref{thm:unique_currents}}]
Let $\beta>\beta'>\beta_c$. Let $c>0$ be the constant of Proposition~\ref{prop:uniform connectivity} and for $L>0$, let $C_0=(2d/c)^{\frac{1}{d-1}}$  and \begin{equation*}
\ell=\ell_L:=C_0 (\log L)^{\frac{1}{d-1}}.
\end{equation*}
Take $L$ large enough so that $\ell_L\leq L$. By Proposition~\ref{prop:uniform connectivity} and a union bound,    
\begin{equation*}
\inf_{A,B \subset \partial \Lambda_{10L}}\mathbf{P}^{A, B}_{\widetilde{\Lambda_{10L}}[\beta,\beta'],\widetilde{\Lambda_{10L}}[\beta,\beta']}[\mathcal{A}_L]\geq 1-L^{-d}.    
\end{equation*}
Hence, by Proposition~\ref{prop:coupling},
\begin{equation*}
\inf_{A,B \subset \partial \Lambda_{10L}}\mathbf{P}^{A, B}_{\Lambda_{10L}[\beta],\Lambda_{10L}[\beta]}[U(L)]\geq (1-L^{-d})\inf_{\omega\in \mathcal{A}_L}\mathbb{P}_{\varepsilon}[\omega\cup \gamma \in U(L)],   
\end{equation*}
which tends to $1$ as $L$ tends to infinity by Proposition~\ref{prop:uniqueness with sprinkling}. This concludes the proof.
\end{proof}

\section{Proof of the sprinkling property}\label{sec:coupling}
Before starting the proof, we introduce one notation. Let $\Lambda\subset \mathbb Z^d$ and $0<\beta'<\beta$. Define the weighted graph $\Lambda[\beta,\beta']=(\Lambda,E,J)$ as follows: $E$ is made of two copies of $E(\Lambda)$ which, given $xy\in E(\Lambda)$, we call $xy(\beta')$ and $xy(\beta-\beta')$ respectively; for each $xy\in E(\Lambda)$, we let $J_{xy(\beta')}=\beta'$ and $J_{xy(\beta-\beta')}=\beta-\beta'$. Below, we use the observation that the law of $\hat{\n}_1$ where $\n_1$ is sampled according to $\mathbf{P}^S_{\Lambda[\beta]}$ is the same as that of $\hat{\n}_2$, where $\n_2$ is the projection onto $E(\Lambda)$ of a sample of $\mathbf{P}^S_{\Lambda[\beta,\beta']}$.
\begin{proof}[Proof of Proposition~\textup{\ref{prop:coupling}}]
Let $\beta,\beta', \Lambda, S$ be as in the statement and recall that $\xi>0$ is chosen so that $\tanh(\xi)^2=\tanh(\beta-\beta')$. We begin by defining two maps $f$ and $g$ on $\tilde{\Lambda}[\beta,\beta']$ and $\Lambda[\beta,\beta']$ respectively.
\begin{enumerate}
    \item[$\bullet$] Fix $\m\in \Omega_{\tilde{\Lambda}[\beta,\beta']}$. Let $\tilde{\m}=f(\m)$ be the current on $\tilde{\Lambda}[\beta,\beta']$ defined as follows: for every $xy\in E(\Lambda)$, $\tilde{\m}(xy)=\m(xy)$; and (recall that $z=z(xy)$) $\tilde{\m}(xz)=\tilde{\m}(yz)=0$ if $\m(xz)$ and $\m(yz)$ are even, and $\tilde{\m}(xz)=\tilde{\m}(yz)=1$ if $\m(xz)$ and $\m(yz)$ are odd (note that $\m(xz)$ and $\m(yz)$ must have the same parity due to sources constraints).
    \item[$\bullet$] Fix $\n\in \Omega_{\Lambda[\beta,\beta']}$. Let $\tilde{\n}=g(\n)$ be the current on $\Lambda[\beta,\beta']$ defined as follows: for every $xy\in E(\Lambda)$, $\tilde{\n}(xy(\beta'))=\n(xy(\beta'))$, and $\tilde{\n}(xy(\beta-\beta'))=\mathds{1}\{\n(xy(\beta-\beta'))\textup{ is odd}\}$. 
\end{enumerate}
Consider the normalised weights
$$
w(\tilde{\n})=\frac{\sum_{\n\in g^{-1}(\tilde{\n})}w(\n)}{\cosh(\beta-\beta')^{|E(\Lambda)|}},
$$
and 
$$
w(\tilde{\m})=\frac{\sum_{\m\in f^{-1}(\tilde{\m})}w(\m)}{\cosh(\xi)^{2|E(\Lambda)|}}
$$
of $\tilde{\n}$ and $\tilde{\m}$, and note that we can rewrite them in the form
$$
w(\tilde{\n})=\prod_{xy\in E(\Lambda)} \frac{(\beta')^{\tilde{\n}(xy(\beta'))}}{\tilde{\n}(xy(\beta'))!} \tanh(\beta-\beta')^{\tilde{\n}(xy(\beta-\beta'))},
$$
and 
$$
w(\tilde{\m})=\prod_{xy\in E(\Lambda)} \frac{(\beta')^{\tilde{\m}(xy(\beta'))}}{\tilde{\m}(xy(\beta'))!} \tanh(\xi)^{\tilde{\m}(xz(\xi))} \tanh(\xi)^{\tilde{\m}(yz(\xi))}.
$$
By the definition of $f$ above and due to the source constraints (the midpoints $z=z(xy)$ of $\tilde{\Lambda}[\beta,\beta']$ are not sources), we have $\tilde{\m}(xz)=\tilde{\m}(yz)$. Hence, our choice $\tanh(\xi)^2=\tanh(\beta-\beta')$ implies that if 
\begin{equation}\label{eq:currents equality}
\tilde{\n}(xy(\beta'))=\tilde{\m}(xy(\beta')) \quad \text{and} \quad \tilde{\n}(xy(\beta-\beta'))=\tilde{\m}(xz)(=\tilde{\m}(yz)) \quad \forall\, xy\in E(\Lambda),    
\end{equation}
then $w(\tilde{\n})=w(\tilde{\m})$. This in turn implies that we can find a coupling $(\tilde{\n},\tilde{\m},\mathbb{Q})$ between the measures $f(\mathbf{P}^{S}_{\Lambda[\beta,\beta']})$ and $g(\mathbf{P}^{S}_{\overline{\Lambda}[\beta,\beta']})$ such that \eqref{eq:currents equality} holds $\mathbb{Q}$-almost surely. 

We now obtain a coupling between the associated percolation configurations $\hat{\n}$ and $\hat{\m}$. To this end, note that conditionally on $\tilde{\n}$, the collection of random variables $(\mathbbm{1}\{\n(xy(\beta-\beta'))>0\})_{xy\in E(\Lambda)}$ is a sequence of independent Bernoulli random variables that take the value $1$ if $\tilde{\n}(xy(\beta-\beta'))=1$ and are otherwise equal to $1$ with a certain probability $p_1$, which does not depend on the edge $xy$. Similarly, conditionally on $\tilde{\m}$, the collection of random variables $(\mathbbm{1}\{\m(xz(xy))\m(yz(xy))>0\})_{xy\in E(\Lambda)}$ is a sequence of independent Bernoulli random variables that take the value $1$ if $\tilde{\m}(xz(xy))=1$ and otherwise take the value $1$ with a certain probability $p_2$. It suffices to show that 
\begin{equation}\label{eq: p_1>p_2}
p_1>p_2.
\end{equation}
Indeed, assume \eqref{eq: p_1>p_2}. Using the coupling between $\tilde{\n}$ and $\tilde{\m}$ that satisfies \eqref{eq:currents equality}, we can couple the random variables $(\mathbbm{1}\{\n(xy(\beta-\beta'))>0\})_{xy\in E(\Lambda)}$ and $(\mathbbm{1}\{\m(xz(\xi))\m(yz(\xi))>0\})_{xy\in E(\Lambda)}$ so that the former is equal to the union of the latter with an independent Bernoulli percolation on $E(\Lambda)$ of parameter $\varepsilon=\frac{p_1-p_2}{1-p_2}$. The existence of the desired coupling between $\hat{\n}$ and $\hat{\m}$ then follows.

Let us now prove \eqref{eq: p_1>p_2}. Using the identity $\rm{sech}^2(\xi):=\frac{1}{\cosh(\xi)^2}=1-\tanh(\xi)^2$, we get
\[
p_2=\left(1-\frac{1}{\cosh(\xi)}\right)^2=1-2\sqrt{1-\tanh(\xi)^2}+1-\tanh(\xi)^2.
\]
Let $t:=\beta-\beta'$. One has
\[
p_1=1-\frac{1}{\cosh(t)}.
\]
Since $\tanh(\xi)^2=\tanh(t)$ by our choice of $\xi$, $1-\tanh(t)=e^{-t}/\cosh(t)$, and $4e^{-t}\cosh(t)=2+2e^{-2t}$,
\[
p_1-p_2=-\frac{1}{\cosh(t)}+\sqrt{\frac{4e^{-t}}{\cosh(t)}}-\frac{e^{-t}}{\cosh(t)}=\frac{1}{\cosh(t)}\left(\sqrt{2+2e^{-2t}}-1-e^{-t}\right).
\]
Now, for $t>0$,
\[
2+2e^{-2t}-(1+e^{-t})^2=1-2e^{-t}+e^{-2t}>0,
\]
which implies that $p_1>p_2$. This completes the proof.
\end{proof}

\section{Proof of the surface-order connectivity}\label{sec:surface connection}

We start by proving the Proposition~\ref{prop:uniform connectivity} for $B=\emptyset$. 
We emphasise that for the argument to work, the first current can have essentially any set of sources, but the second current needs to have no sources. We will later see how to deal with the sources of the second current.

\begin{Prop}
 	\label{prop:connectivity_old}
 	Let $\beta>\beta'>\beta_c$. There exists $c>0$ such that, for every $N \geq 2n$, every $x, y \in \Lambda_{N}$ and every $S\subset \Lambda_{N}$ of even cardinality with $S\cap (\Lambda_n(x) \cup \Lambda_n(y))=\emptyset$, we have
 	$$\mathbf{P}^{S, \emptyset}_{\widetilde{\Lambda_N}[\beta,\beta'], \widetilde{\Lambda_N}[\beta,\beta']}\Big[\lr{}{\widehat{\n_1 + \n_2}\:}{\Lambda_n(x)}{\Lambda_n(y)}\Big] \geq 1 - e^{-cn^{d-1}}\,.$$
 \end{Prop}

In order to prove this result, we first need to introduce some additional notation.
Consider the weighted graph 
$\Lambda_N^{\bullet\bullet}:=\Lambda_N^{\bullet\bullet}(x,y,n,\beta,\beta')$ obtained from $\widetilde{\Lambda_N}[\beta,\beta']$ by replacing $\Lambda_{n/2}(x)$ and $\Lambda_{n/2}(y)$ by two vertices ${\bf x}$ and ${\bf y}$, which are connected to each vertex outside $\Lambda_{n/2}(x)\cup
 \Lambda_{n/2}(y)$ by as many edges as the number of edges between this vertex and $\Lambda_{n/2}(x)$ or $\Lambda_{n/2}(y)$. Similarly, we define $\Lambda_{n}^\bullet(x)$ to be the weighted graph obtained from $\Lambda_{n}(x)[\beta]$ by merging all the vertices in $\Lambda_{{n/2}}(x)$ into a single vertex.
For the proof, we also need the following lemma from \cite{DGR20}. 

\begin{Lem}[\hspace{1pt}{\cite[Lemma 4.2]{DGR20}}]\label{lem:weak weak mix}
Let $\beta>\beta_c$ and $\varepsilon>0$. For every $n$ large enough and every $A\subset\partial\Lambda_n$ of even cardinality,
\begin{equation}\label{eq:wwmixing}\langle{\sigma_A}\rangle_{\Lambda_n^\bullet} \leq e^{\varepsilon n^{d-1}}\langle {\sigma_A}\rangle_{\Lambda_n[\beta]}.\end{equation}
\end{Lem}
Let us give a short  heuristic for why \eqref{eq:wwmixing} holds. Note that by the Markov property we have that
\begin{equation*}
\langle{\sigma_A}\rangle_{\Lambda_n^\bullet} = \langle{\sigma_A} \mid \sigma_x=1 \,\forall x\in\partial\Lambda_{n/2} \rangle_{\Lambda_n[\beta]}.
\end{equation*}
 By a straightforward Hamiltonian comparison, we have  $\langle{\sigma_A} \mid \sigma_x=1 \,\forall x\in\partial\Lambda_{n/2} \rangle_{\Lambda_n[\beta]} \leq e^{C n^{d-1}}\langle {\sigma_A}\rangle_{\Lambda_n[\beta]}$ for a large enough constant $C$. The constant $C$ can be turned into a small constant $\varepsilon>0$ by arguing that the effect of the conditioning of each point in $\partial \Lambda_{n/2}$ can be made small. This is proved by using the fact that even correlations can be written in terms of pairing events for the corresponding random cluster measure, and that the infinite volume random cluster measure is unique.

We now prove Proposition~\ref{prop:connectivity_old}, whose proof is similar to that of \cite[Proposition 4.1]{DGR20}.
 
\begin{proof}[Proof of Proposition~\textup{\ref{prop:connectivity_old}}]
If $\Lambda_n(x)\cap\Lambda_n(y)\ne\emptyset$, then the result is straightforward, so we assume that $\Lambda_n(x)\cap\Lambda_n(y)=\emptyset$. 
Consider the graph $\Lambda_N^{\bullet\bullet}$ defined above.
By the switching lemma (see Lemma \ref{lem:switching}), 
\begin{equation}
	\label{eq:collapsed_dc1}
	\mathbf{P}^{S, \emptyset}_{\Lambda_N^{\bullet\bullet}, \Lambda_N^{\bullet\bullet}}[\lr{}{\widehat{\n_1 + \n_2}\:}{{\bf x}}{{\bf y}}]=\frac{\langle{\sigma_{{\bf x}}\sigma_{{\bf y}}\sigma_S}\rangle_{\Lambda_N^{\bullet\bullet}}\,\langle{\sigma_{{\bf x}}\sigma_{{\bf y}}}\rangle_{\Lambda_N^{\bullet\bullet}}}{\langle{\sigma_S}\rangle_{\Lambda_N^{\bullet\bullet}}}\ge \langle{\sigma_{{\bf x}}\sigma_{{\bf y}}}\rangle_{\Lambda_N^{\bullet\bullet}}^2,
\end{equation}
	where in the second inequality we used Griffiths' inequalities (see Proposition~\ref{prop: Griffiths}). By the Edwards--Sokal coupling (see Section \ref{sec:preliminaries}) we have
	$$
    \langle{\sigma_{{\bf x}}\sigma_{{\bf y}}}\rangle_{\Lambda_N^{\bullet\bullet}}=\phi_{\Lambda_N^{\bullet\bullet}}[{\bf x}\longleftrightarrow{\bf y}].
    $$
    Note that we can view $\phi_{\Lambda_N^{\bullet\bullet}}$ as the measure $\phi_{\widetilde{\Lambda_N}[\beta,\beta']}$ with edges in $\Lambda_{n/2}(x)\cup\Lambda_{n/2}(y)$ conditioned to be open (an event we call $\mathcal{E}_1$), and similarly we can view $\phi_{\Lambda_N[\beta']}$ as the measure $\phi_{\widetilde{\Lambda_N}[\beta,\beta']}$ with edges not labelled as $\beta'$ conditioned to be closed (an event we call $\mathcal{E}_2$).
    Thus, applying the FKG inequality twice we obtain
    \begin{align}\label{eq:k}
    \begin{aligned}\langle{\sigma_{{\bf x}}\sigma_{{\bf y}}}\rangle_{\Lambda_N^{\bullet\bullet}}&=\phi_{\widetilde{\Lambda_N}[\beta,\beta']}[\Lambda_{n/2}(x)\longleftrightarrow \Lambda_{n/2}(y)\: |\: \mathcal{E}_1]\\&\ge \phi_{\widetilde{\Lambda_N}[\beta,\beta']}[\Lambda_{n/2}(x)\longleftrightarrow \Lambda_{n/2}(y)]
    \\&\ge \phi_{\widetilde{\Lambda_N}[\beta,\beta']}[\Lambda_{n/2}(x)\longleftrightarrow \Lambda_{n/2}(y)\: |\: \mathcal{E}_2]
    \\&= \phi_{\Lambda_N[\beta']}[\Lambda_{{n/2}}(x)\longleftrightarrow\Lambda_{{n/2}}(y)]\ge 1-\exp(-cn^{d-1}),
    \end{aligned}\end{align}
    for some constant $c>0$,
	where the last inequality follows from \cite{Bod05} as explained in Remark \ref{rem: connectivity fk}.
	
	Now, let us prove that for every $n$ large enough we have that
	\begin{equation}\label{eq:kk}
    \mathbf{P}^{S, \emptyset}_{\widetilde{\Lambda_N}[\beta,\beta'],\widetilde{\Lambda_N}[\beta,\beta']}[\mathcal E]\le e^{c' n^{d-1}}\mathbf{P}^{S, \emptyset}_{\Lambda_N^{\bullet\bullet}, \Lambda_N^{\bullet\bullet}}[\mathcal E]
    \end{equation}
    for some constant $c'<c/2$,
    where $\mathcal{E}=\{\lr{}{\widehat{\n_1 + \n_2}\:}{\Lambda_{n}(x)}{\Lambda_{n}(y)}\}^c$.
	This would be enough to conclude, since we can combine \eqref{eq:k} and \eqref{eq:collapsed_dc1}, and then use \eqref{eq:kk} to obtain that 
	$$\mathbf{P}^{S, \emptyset}_{\widetilde{\Lambda_N}[\beta,\beta'], \widetilde{\Lambda_N}[\beta,\beta']}[\mathcal{E}]\le e^{c' n^{d-1}}\mathbf{P}^{S, \emptyset}_{\Lambda_N^{\bullet\bullet}, \Lambda_N^{\bullet\bullet}}[\mathcal{E}]\le e^{c'n^{d-1}}\mathbf{P}^{S,\emptyset}_{\Lambda_N^{\bullet\bullet},\Lambda_N^{\bullet\bullet}}[\{\lr{}{\widehat{\n_1+\n_2}\:}{\mathbf{x}}{\mathbf{y}}\}^c]\leq {2} e^{-c' n^{d-1}}.$$
We therefore focus on proving \eqref{eq:kk}. Let $\Lambda \coloneqq \Lambda_n(x) \cup \Lambda_n(y)$ and let $E$ denote the set of edges in $\widetilde{\Lambda_N}[\beta,\beta']$ that do not lie in $\tilde{\Lambda}$. Consider a current ${\bf m}\in \mathbb{N}^E$ with $\partial {\bf m}\subset\partial\Lambda$ and define $\mathcal E_{\bf m}$ to be the event
that $\n$ is equal to ${\bf m}$ on $E$. Rewriting the weights we obtain that 
\begin{align*}
\frac{\mathbf{P}^{\emptyset}_{\widetilde{\Lambda_N}[\beta,\beta']}[\mathcal E_{\bf m}]}{\mathbf{P}^{\emptyset}_{\Lambda_N^{\bullet\bullet}}[\mathcal E_{\bf m}]}&=\frac{\langle{\sigma_{\partial{\bf m}}}\rangle_{(\Lambda_{n}(x)\cup\Lambda_{n}(y))[\beta]}}{\langle{\sigma_{\partial{\bf m}}}\rangle_{\Lambda_{n}^\bullet(x)\cup\Lambda_{n}^\bullet(y)}}\,\frac{\displaystyle\sum_{\m' \in  \mathbb{N}^E: \, \partial\m'\subset\partial\Lambda}w_\beta(\m')\langle{\sigma_{\partial\m'}}\rangle_{\Lambda_{n}^\bullet(x)\cup\Lambda_{n}^\bullet(y)}}{\displaystyle\sum_{\m' \in  \mathbb{N}^E: \, \partial\m'\subset\partial\Lambda}w_\beta(\m')\langle{\sigma_{\partial\m'}}\rangle_{(\Lambda_{n}(x)\cup\Lambda_{n}(y))[\beta]}}\\
&\le \max_{\substack{A\subset\partial\Lambda_{n}(x)\\ |A|\text{ even}}}\frac{\langle{\sigma_A}\rangle_{\Lambda_{n}^\bullet(x)}}{\langle{\sigma_A}\rangle_{\Lambda_{n}(x)[\beta]}}\times\max_{\substack{B\subset\partial\Lambda_{n}(y)\\ |B|\text{ even}}}\frac{\langle{\sigma_B}\rangle_{\Lambda_{n}^\bullet(y)}}{\langle{\sigma_B}\rangle_{\Lambda_{n}(y)[\beta]}},
\end{align*}
where we used Griffiths' inequality (see Proposition~\ref{prop: Griffiths}) to argue that $\langle{\sigma_{\partial{\bf m}}}\rangle_{(\Lambda_{n}(x)\cup\Lambda_{n}(y))[\beta]}\leq \langle{\sigma_{\partial{\bf m}}}\rangle_{\Lambda_{n}^\bullet(x)\cup\Lambda_{n}^\bullet(y)}$.
A similar bound can be derived for $\mathbf{P}^{S}_{\Lambda_N[\beta,\beta']}$ and $\mathbf{P}^{S}_{\Lambda_N^{\bullet\bullet}}$ instead of $\mathbf{P}^{\emptyset}_{\Lambda_N[\beta,\beta']}$ and $\mathbf{P}^{\emptyset}_{\Lambda_N^{\bullet\bullet}}$. We deduce, by decomposing on possible values of $(\n_1,\n_2)$ on $E$ that satisfy the event $\mathcal{E}$, 
$$\mathbf{P}^{S, \emptyset}_{\widetilde{\Lambda_N}[\beta,\beta'], \widetilde{\Lambda_N}[\beta,\beta']}[\mathcal E]\le  \Bigg(\max_{\substack{A\subset\partial\Lambda_{n}\\ |A|\text{ even}}}\frac{\langle{\sigma_A}\rangle_{\Lambda_{n}^\bullet}}{\langle{\sigma_A}\rangle_{\Lambda_{n}[\beta]}}\Bigg)^4\ \mathbf{P}^{S, \emptyset}_{\Lambda_N^{\bullet\bullet}, \Lambda_N^{\bullet\bullet}}[\mathcal E],$$
so that \eqref{eq:kk} follows from Lemma~\ref{lem:weak weak mix}.
\end{proof}

To conclude the proof of Proposition~\ref{prop:uniform connectivity}, we show that an arbitrary set of sources can be added to the second current in the statement of Proposition~\ref{prop:connectivity_old}. The proof will use a monotonicity property of random currents with respect to the set of sources first derived in \cite{duminil2021conformal}. We stress that in order apply this monotonicity property, we need to work with annulus-crossing events instead of the box-to-box connection events considered in \cite{DGR20}.

\begin{proof}[Proof of Proposition~\textup{\ref{prop:uniform connectivity}}]
The result holds trivially if $\Lambda_n(x)\cap \partial \Lambda_{N-2n}\neq \emptyset$, so let us assume that $\Lambda_n(x)\cap \partial \Lambda_{N-2n}= \emptyset$.
It follows from Proposition~\ref{prop:connectivity_old} that     
$$\mathbf{P}^{A, \emptyset}_{\widetilde{\Lambda_N}[\beta,\beta'], \widetilde{\Lambda_N}[\beta,\beta']}[\lr{}{\widehat{\n_1 + \n_2}\:}{\Lambda_n(x)}{\partial \Lambda_{N-2n}}] \geq 1 - e^{-cn^{d-1}}\,.$$
To conclude, it therefore suffices to show that 
\begin{equation}\label{eq:current monotonicity}
\mathbf{P}^{A, B}_{\widetilde{\Lambda_N}[\beta,\beta'], \widetilde{\Lambda_N}[\beta,\beta']}[\lr{}{\widehat{\n_1 + \n_2}\:}{\Lambda_n(x)}{\partial \Lambda_{N-2n}}] \geq \mathbf{P}^{A, \emptyset}_{\widetilde{\Lambda_N}[\beta,\beta'], \widetilde{\Lambda_N}[\beta,\beta']}[\lr{}{\widehat{\n_1 + \n_2}\:}{\Lambda_n(x)}{\partial \Lambda_{N-2n}}].
\end{equation}
A similar statement appeared in \cite[Lemma~3.3]{duminil2021conformal}. Let $\mathsf{C}$ denote the union of all clusters that intersect $\Lambda_n(x)$. Consider a set of vertices $C$ (of $\widetilde{\Lambda}_N[\beta,\beta']$) that contains $\Lambda_n(x)$ and does not intersect the annulus $\Lambda_N\setminus\Lambda_{N-2n-1}$ (i.e. a potential realisation of $\mathsf{C}$ under the event $\{\lr{}{\widehat{\n_1 + \n_2}\:}{\Lambda_n(x)}{\partial \Lambda_{N-2n}}\}^c$). Write $\mathbf{C}$ for the set of vertices that lie either in $C$ or are adjacent to an element of $C$. Note that we may view $\mathbf{C}$ as a sub-weighted graph of $\widetilde{\Lambda}_N[\beta,\beta']$. Then, 
\begin{equation}\label{eq:prob cluster}
\mathbf{P}^{A, B}_{\widetilde{\Lambda_N}[\beta,\beta'], \widetilde{\Lambda_N}[\beta,\beta']}[\mathsf{C}=C]=\frac{Z^{\emptyset,\emptyset}_{\mathbf{C},\mathbf{C}}[\mathsf{C}=C]Z^{A}_{\widetilde{\Lambda_N}[\beta,\beta']\setminus \mathbf{C}}Z^{B}_{\widetilde{\Lambda_N}[\beta,\beta']\setminus \mathbf{C}}}{Z^{A}_{\widetilde{\Lambda_N}[\beta,\beta']}Z^{B}_{\widetilde{\Lambda_N}[\beta,\beta']}}.
\end{equation}
Recall that 
\[Z^{B}_{\widetilde{\Lambda_N}[\beta,\beta']}=2^{-|\widetilde{\Lambda_N}[\beta,\beta']|}Z^{\textup{Ising}}_{\widetilde{\Lambda_N}[\beta,\beta']}[\sigma_B] \]
and
\[Z^{B}_{\widetilde{\Lambda_N}[\beta,\beta']\setminus \mathbf{C}}=2^{-|\widetilde{\Lambda_N}[\beta,\beta']|+|\mathbf{C}|}Z^{\textup{Ising}}_{\widetilde{\Lambda_N}[\beta,\beta']\setminus \mathbf{C}}[\sigma_B] .
\] 
We can express $Z^{\textup{Ising}}_{\widetilde{\Lambda_N}[\beta,\beta']}[\sigma_B]$ as $Z^{\textup{Ising}}_{\widetilde{\Lambda_N}[\beta,\beta']\setminus \mathbf{C}}[\sigma_B G(\sigma)]$, where 
\[
G(\sigma)=\sum_{\tau\in \{\pm 1\}^{\mathbf{C}}}\prod_{xy\in E(\mathbf{C})} e^{J_{xy} \tau_x\tau_y}\prod_{xy\in \partial_E \mathbf{C}} e^{J_{xy}\tau_x\sigma_y},
\]
where $J$ is defined as above Figure \ref{fig:doubleedge} and $\partial_E \mathbf{C}$ denotes the edges of $\widetilde{\Lambda_N}[\beta,\beta']$ with one endpoint in $\mathbf{C}$ and one endpoint outside. By Taylor expanding $e^{J_{xy}\tau_x\sigma_y}$ and using Griffiths' inequality (see Proposition~\ref{prop: Griffiths}), we obtain that $\langle \sigma_B G(\sigma)\rangle_{\widetilde{\Lambda_N}[\beta,\beta']\setminus\mathbf{C}}\geq \langle \sigma_B \rangle_{\widetilde{\Lambda_N}[\beta,\beta']\setminus \mathbf{C}}\langle G(\sigma)\rangle_{\widetilde{\Lambda_N}[\beta,\beta']\setminus \mathbf{C}}$. 
Thus,
\begin{align*}
\frac{Z^{B}_{\widetilde{\Lambda_N}[\beta,\beta']}}{Z^{B}_{\widetilde{\Lambda_N}[\beta,\beta']\setminus \mathbf{C}}}&=2^{|\mathbf{C}|}\frac{\langle \sigma_B G(\sigma)\rangle_{\widetilde{\Lambda_N}[\beta,\beta']\setminus \mathbf{C}}}{\langle \sigma_B\rangle_{\widetilde{\Lambda_N}[\beta,\beta']\setminus \mathbf{C}}}\\&\geq 2^{|\mathbf{C}|} \langle G(\sigma)\rangle_{\widetilde{\Lambda_N}[\beta,\beta']\setminus \mathbf{C}}=\frac{Z^{\emptyset}_{\widetilde{\Lambda_N}[\beta,\beta']}}{Z^{\emptyset}_{\widetilde{\Lambda_N}[\beta,\beta']\setminus \mathbf{C}}}.
\end{align*}
Combining the above inequality with \eqref{eq:prob cluster} we deduce that
\[
\mathbf{P}^{A, B}_{\widetilde{\Lambda_N}[\beta,\beta'], \widetilde{\Lambda_N}[\beta,\beta']}[\mathsf{C}=C]\leq \mathbf{P}^{A, \emptyset}_{\widetilde{\Lambda_N}[\beta,\beta'], \widetilde{\Lambda_N}[\beta,\beta']}[\mathsf{C}=C].
\]
Summing over all $C$ that contain $\Lambda_n(x)$ and do not intersect the annulus $\Lambda_N\setminus\Lambda_{N-2n-1}$, and then taking complements, we obtain \eqref{eq:current monotonicity}, which completes the proof.
\end{proof}

\section{Proof of uniqueness via sprinkling}\label{sec:proof uniqueness with sprinkling}

The statement of Proposition~\ref{prop:uniqueness with sprinkling} is almost the same as that of \cite[Proposition 2.2]{SeveroSlabIsing2023}. The only subtle difference that the latter is concerned with the ``uniqueness with sprinkling event''  (namely that every crossing cluster of $\omega$ belongs to the same $(\omega\cup\gamma)$-cluster) which depends on the pair $(\omega, \omega\cup\gamma)$, whereas here we deal with the ``uniqueness without sprinkling event'' $U(L)$ (i.e.~there is exactly one crossing cluster in $\omega\cup\gamma$) which only depends on the (sprinkled) configuration $\omega\cup\gamma$. Below, we deduce the latter from the former via an extra sprinkling and a simple exploration argument.

\begin{proof}[Proof of Proposition~\textup{\ref{prop:uniqueness with sprinkling}}]
    First note that the existence of a cluster crossing $\Lambda_{8 L}\setminus \Lambda_{ L}$ is  implied by the fact that $\omega\in\mathcal{A}_L$, so we only need to focus on the uniqueness part. 

    Enlarge the probability space and write $\gamma=\gamma^1\cup\gamma^2$, with $\gamma^1$ and $\gamma^2$ independent Bernoulli percolations with of parameter $\delta$ given by $1-\varepsilon= (1-\delta)^2$. It follows readily from (the proof of) \cite[Proposition 2.2]{SeveroSlabIsing2023} that 
    \begin{equation}\label{eq:uniqueness with sprinkling}
    	\lim_{L\to\infty} \inf_{\omega\in\mathcal{A}_L(\ell_L)} \mathbb{P}_\varepsilon \left[ \begin{array}{c}
    		\text{all clusters of } \omega \text{ crossing } \Lambda_{4L}\setminus\Lambda_{2L} \\ \text{ are connected to each other in } \omega\cup\gamma^1
    	\end{array} \right] = 1.
    \end{equation}
    
    We now prove that, conditionally on the event in \eqref{eq:uniqueness with sprinkling}, we have $\omega\cup\gamma\in U(L)$ with high probability, thus concluding the proof. Condition on a realization of $\gamma^1$ satisfying the event in \eqref{eq:uniqueness with sprinkling}, and denote by $C_*$ the cluster of $\omega\cup\gamma^1$ containing all the clusters of $\omega$ crossing $\Lambda_{4L}\setminus\Lambda_{2L}$. Fix a vertex $x\in \Lambda_{2L}$ and explore the cluster of $x$ in $\omega\cup\gamma = \omega\cup\gamma^1\cup\gamma^2$ in the standard way: by revealing, at each step, the state of $\gamma^2$ on an edge at the boundary of the current cluster, and adding that edge to the cluster if and only if it is open in $\omega\cup\gamma$. Notice that if the cluster of $x$ crosses the annulus $\Lambda_{4L}\setminus\Lambda_{2L}$, then the exploration reaches each annulus of the form $\Lambda_{2L + i2\ell_L}\setminus\Lambda_{2L + (i-1)2\ell_L}$, for $i\in\{1,2,\ldots, \lfloor L/\ell_L \rfloor$\}. Since $\omega\in\mathcal{A}_L(\ell_L)$, each time the exploration reaches such an annulus for the first time, it is at distance at most $2\ell_L$ from $C_*$, and therefore gets connected to $C_*$ within this annulus with probability at least $\delta^{2\ell_L}$. As a result, since $\ell_L=o(\log L)$, the probability that the cluster of $x$ crosses $\Lambda_{4L}\setminus\Lambda_{2L}$ without being connected to $C_*$ is at most $(1-\delta^{2\ell_L})^{\lfloor L/\ell_L \rfloor} \leq e^{-cL^{1-o(1)}}$. By a union bound in $x$, this yields that $\omega\cup\gamma\in U(L)$ with high probability, thus concluding the proof. 
\end{proof}

\section{Proof of the corollaries}\label{sec:corollaries}

\begin{proof}[Proof of Corollary~\textup{\ref{coro: truncated}}] Let $\beta>\beta_c$ and $n\geq 1$.
Recall the coarse-grained percolation model $\eta^L$ constructed below Theorem~\ref{thm:unique_currents} and import all the notations from there. We fix $L$ large enough so that $\eta^L$ stochastically dominates a Bernoulli site percolation with parameter $p$, which in turn is chosen close enough to $1$ (namely $p>1-p_c^{\textup{site}}(\Z^d)$) so that the existence of closed paths decays exponentially fast in the distance.

$\bullet$ \textbf{Step 1.} We start by proving \eqref{eq:truncated} for $A$ of even cardinality. 
Let $N>2n$ to be taken large enough. By Corollary~\ref{coro:expression truncated switching} we have
 \begin{equation}\label{eq:truncated proof}
     \langle\sigma_A;\sigma_B\rangle_{\Lambda_N[\beta]}^+
     = \langle\sigma_A\sigma_B\rangle_{\Lambda_N[\beta]}^+ \mathbf{P}^{A\cup B, \emptyset}_{\Lambda_N^\fg[\beta],\Lambda_N^\fg[\beta]}[\mathcal{F}_A^{\mathrm{c}}].
 \end{equation}
Due to the sources constraint and the parity of $|A|$, the event $\mathcal{F}_A^{\mathrm{c}}$ implies the existence of two disjoint clusters connecting $A$ and $B$, thus in particular crossing the annulus $\Lambda_{2n}\setminus\Lambda_n$. However, this cannot happen if there exists a blocking surface (i.e.~a cutset) $\eta^L$-open sites separating the inner and outer boundaries of the annulus. By the aforementioned comparison with Bernoulli percolation, the latter event holds with probability at least $1-e^{-cn}$, thus concluding the proof of \eqref{eq:truncated}.

$\bullet$ \textbf{Step 2.} We now turn to \eqref{eq:truncated_pt_to_pt}. Fix $N>|x|$, so that $x\in\Lambda_N$. We have
 \begin{equation}\label{eq:truncated_proof}
     \langle\sigma_0;\sigma_x\rangle^+_{\Lambda_N[\beta]} = \frac{Z^{\{0,x\}}_{\Lambda_N^\fg[\beta]} Z^\emptyset_{\Lambda_N^\fg[\beta]} - Z^{\{0,\fg\}}_{\Lambda_N^\fg[\beta]} Z^{\{x,\fg\}}_{\Lambda_N^\fg[\beta]}}{(Z^\emptyset_{\Lambda_N^\fg[\beta]})^2} = \langle\sigma_0\sigma_x\rangle^+_{\Lambda_N[\beta]} \mathbf{P}^{\{0,x\}, \emptyset}_{\Lambda_N^\fg[\beta],\Lambda_N^\fg[\beta]}[\nlr{}{}{0}{\fg}],
 \end{equation}
where the last equality follows readily from Lemma~\ref{lem:switching}. 
Due to the comparison with a highly supercritical Bernoulli percolation, with probability at least $1-e^{-c|x|}$, there exists an $\eta^L$-open blocking surface separating\footnote{Recall the definition, for a set $S\subset \mathbb Z^d$, of $S_L=\{z\in\Lambda_{\tilde{N}}:~\Lambda_{10 L}(Lz)\cap S\neq\emptyset\}$, introduced below Theorem~\ref{thm:unique_currents}.}  $\{0\}_L$ and $\{x\}_L$,
which in turn is also connected to $\partial \Lambda_{\tilde{N}}$ (where we recall that $\tilde{N}=\lfloor N/L\rfloor-10$) in at least $c \tilde{N}^{d-1}\gg |x|$ many points for some $c>0$. Conditionally on this event, a simple insertion tolerance argument gives that the $\eta$-cluster induced by this $\eta^L$-cluster is connected to $\fg$ with probability at least $1-e^{-c|x|}$. However, on this event, any path from $0$ to $x$ in $\eta$ (which exists due to the source constraint) must be connected to $\fg$. This yields that $\mathbf{P}^{\{0,x\}, \emptyset}_{\Lambda_N^\fg[\beta],\Lambda_N^\fg[\beta]}[\nlr{}{}{0}{\fg}] \leq e^{-c|x|}$, concluding the proof of \eqref{eq:truncated_pt_to_pt}.

$\bullet$ \textbf{Step 3.} Finally, we use Steps 1 and 2 to prove the exponential ratio mixing for $A$ of odd cardinality. Fix $c=c(\beta,d)>0$ which satisfies the conclusions of Steps 1 and 2. Let $A\subset \Lambda_n$ of odd cardinality, $x\in A$, and $B\subset \mathbb{Z}^d\setminus \Lambda_{2n}$. By the Lebowitz inequality (see Proposition \ref{prop: lebowitz} and Remark \ref{rem:twist of lebowitz}),
\[
\langle \sigma_A \sigma_B |\: \sigma_x=1\rangle^+_{\Lambda_N[\beta]} - \langle \sigma_A \sigma_B \rangle^+_{\Lambda_N[\beta]}\geq \Big|\langle \sigma_A |\: \sigma_x=1\rangle^+_{\Lambda_N[\beta]}\langle \sigma_B \rangle^+_{\Lambda_N[\beta]}-\langle \sigma_B |\: \sigma_x=1\rangle^+_{\Lambda_N[\beta]}\langle \sigma_A \rangle^+_{\Lambda_N[\beta]}\Big|.
\]
Writing $\mathbbm{1}_{\sigma_x=1}=\frac{\sigma_x+1}{2}$ and multiplying both sides of the above equation by $2\langle \mathds{1}_{\sigma_x=1}\rangle^+_{\Lambda_N[\beta]}$ yields
\begin{equation}\label{eq: Lebowitz}
\langle \sigma_A \sigma_B \sigma_x\rangle^+_{\Lambda_N[\beta]} - \langle \sigma_A \sigma_B \rangle^+_{\Lambda_N[\beta]}\langle \sigma_x \rangle^+_{\Lambda_N[\beta]}\geq \Big|\langle \sigma_A  \sigma_x\rangle^+_{\Lambda_N[\beta]}\langle \sigma_B \rangle^+_{\Lambda_N[\beta]}-\langle \sigma_B \sigma_x\rangle^+_{\Lambda_N[\beta]}\langle \sigma_A \rangle^+_{\Lambda_N[\beta]}\Big|.
\end{equation}
Rearranging and using that $\sigma_x^2=1$ (so that $\sigma_A\sigma_x=\sigma_{A\setminus \{x\}}$ since $x\in A$) we obtain
\[
\langle \sigma_A \sigma_B \rangle^+_{\Lambda_N[\beta]}\langle \sigma_x \rangle^+_{\Lambda_N[\beta]}\leq \langle \sigma_{A\setminus \{x\}}\sigma_B\rangle^+_{\Lambda_N[\beta]}-\langle \sigma_{A\setminus\{x\}}\rangle^+_{\Lambda_N[\beta]}\langle \sigma_B \rangle^+_{\Lambda_N[\beta]}+\langle \sigma_B \sigma_x\rangle^+_{\Lambda_N[\beta]}\langle \sigma_A \rangle^+_{\Lambda_N[\beta]}.
\]
Now, since $A\setminus \{x\}$ is even, by Step 1, 
\[
\langle \sigma_{A\setminus \{x\}}\sigma_B\rangle^+_{\Lambda_N[\beta]}-\langle \sigma_{A\setminus\{x\}}\rangle^+_{\Lambda_N[\beta]}\langle \sigma_B \rangle^+_{\Lambda_N[\beta]}\leq e^{-cn}\langle \sigma_{A\setminus\{x\}}\rangle^+_{\Lambda_N[\beta]}\langle \sigma_B \rangle^+_{\Lambda_N[\beta]}\leq e^{-cn}\frac{\langle \sigma_A\rangle^+_{\Lambda_N[\beta]}\langle \sigma_B \rangle^+_{\Lambda_N[\beta]}}{\langle \sigma_x\rangle^+_{\Lambda_N[\beta]}},
\]
where we used Griffiths' inequality (see Proposition \ref{prop: Griffiths}) in the last inequality. It remains to prove that 
\begin{equation}\label{eq: target}
\langle \sigma_B \sigma_x\rangle^+_{\Lambda_N[\beta]}\leq (1+e^{-cn})\langle \sigma_B \rangle^+_{\Lambda_N[\beta]}\langle \sigma_x\rangle^+_{\Lambda_N[\beta]},
\end{equation}
from which the desired result follows (with potentially a different constant $c$). To this end, take $y\in \Lambda_n$ such that $\lVert y-x \rVert \geq n$ and apply \eqref{eq: Lebowitz} for $A=\{y\}$. By rearranging, 
\begin{equation}\label{eq:last equ proof trunc}
\langle \sigma_B \sigma_x\rangle^+_{\Lambda_N[\beta]}\langle \sigma_y  \rangle^+_{\Lambda_N[\beta]} \leq \langle \sigma_x \sigma_y \sigma_B \rangle^+_{\Lambda_N[\beta]} - \langle \sigma_y \sigma_B \rangle^+_{\Lambda_N[\beta]}\langle \sigma_x \rangle^+_{\Lambda_N[\beta]}+\langle \sigma_x  \sigma_y\rangle^+_{\Lambda_N[\beta]}\langle \sigma_B \rangle^+_{\Lambda_N[\beta]}.
\end{equation}
By  Griffiths' inequality, the exponential ratio mixing we proved above for $A$ of even cardinality in Step 1, and the exponential decay of truncated $2$-point functions from Step 2, 
\begin{align*}
\langle \sigma_x \sigma_y \sigma_B \rangle^+_{\Lambda_N[\beta]} &- \langle \sigma_y \sigma_B \rangle^+_{\Lambda_N[\beta]}\langle \sigma_x \rangle^+_{\Lambda_N[\beta]}
\\&\leq 
\langle \sigma_x \sigma_y \sigma_B \rangle^+_{\Lambda_N[\beta]} - \langle \sigma_y  \rangle^+_{\Lambda_N[\beta]}\langle \sigma_B \rangle^+_{\Lambda_N[\beta]}\langle \sigma_x \rangle^+_{\Lambda_N[\beta]}
\\ 
&\leq \Big(\langle \sigma_x\sigma_y\rangle_{\Lambda_N[\beta]}^+-\langle \sigma_y  \rangle^+_{\Lambda_N[\beta]}\langle \sigma_x \rangle^+_{\Lambda_N[\beta]}\Big)\langle \sigma_B \rangle^+_{\Lambda_N[\beta]}+e^{-cn}\langle \sigma_x\sigma_y\rangle_{\Lambda_N[\beta]}^+\langle \sigma_B\rangle_{\Lambda_N[\beta]}^+
\\
&\leq 3e^{-cn}\langle \sigma_x \rangle^+_{\Lambda_N[\beta]}\langle  \sigma_y \rangle^+_{\Lambda_N[\beta]}\langle \sigma_B \rangle^+_{\Lambda_N[\beta]}.
\end{align*}
Using again Step 2, we obtain 
\[
\langle \sigma_x  \sigma_y\rangle^+_{\Lambda_N[\beta]}\langle \sigma_B \rangle^+_{\Lambda_N[\beta]}\leq (1+e^{-cn})\langle \sigma_x \rangle^+_{\Lambda_N[\beta]}\langle  \sigma_y \rangle^+_{\Lambda_N[\beta]}\langle \sigma_B \rangle^+_{\Lambda_N[\beta]}.
\]
Combining the two previous equations and \eqref{eq:last equ proof trunc} yields \eqref{eq: target}, thus completing the proof.
\end{proof}

Below we obtain exponential mixing for the Ising model in the supercritical regime by using the exponential decay of the truncated 2-point function, the FKG inequality (see Proposition \ref{prop:FKG}), and the classical observation \eqref{eq: increasing}---see e.g.\ \cite[Proof of Theorem 3.28]{FriedliVelenikIntroStatMech2017}.

\begin{proof}[Proof of Corollary~\textup{\ref{cor: ising mixing}}] Let $\beta>\beta_c$ and denote by $\textup{Cov}(\cdot,\cdot)$ the covariance under $\nu^+_\beta=\langle \cdot\rangle^+_\beta$.
It suffices to prove the desired result for events depending on finitely many vertices.
Let $A$ be an event measurable with respect to $\Lambda_n$ and $B$ an event measurable with respect to a finite subset $S$ of $\mathbb{Z}^d\setminus \Lambda_{2n}$. Observe that the functions 
\begin{equation}\label{eq: increasing}
\pm\mathbbm{1}_A +\frac{1}{2}\sum_{x\in \Lambda_n} \sigma_x, \quad \pm\mathbbm{1}_B+\frac{1}{2}\sum_{x\in S} \sigma_x
\end{equation}
are increasing.
By the FKG inequality,
\begin{align*}
0&\leq \textup{Cov}\Big(\mathbbm{1}_A+\frac{1}{2}\sum_{x\in \Lambda_n} \sigma_x , -\mathbbm{1}_B+\frac{1}{2}\sum_{x\in S} \sigma_x\Big)+ \textup{Cov}\Big(- \mathbbm{1}_A+\frac{1}{2}\sum_{x\in \Lambda_n} \sigma_x, \mathbbm{1}_B+\frac{1}{2}\sum_{x\in S} \sigma_x\Big)\\
&=2\textup{Cov}\Big(\frac{1}{2}\sum_{x\in \Lambda_n} \sigma_x,\frac{1}{2}\sum_{x\in S} \sigma_x\Big)-2\textup{Cov}(\mathbbm{1}_A,\mathbbm{1}_B),
\end{align*}
and 
\begin{align*}
0&\leq \textup{Cov}\Big(\mathbbm{1}_A+\frac{1}{2}\sum_{x\in \Lambda_n} \sigma_x ,\mathbbm{1}_B+ \frac{1}{2}\sum_{x\in S} \sigma_x\Big)+ \textup{Cov}\Big(- \mathbbm{1}_A+\frac{1}{2}\sum_{x\in \Lambda_n} \sigma_x, - \mathbbm{1}_B+\frac{1}{2}\sum_{x\in S} \sigma_x\Big)\\
&=2\textup{Cov}\Big(\frac{1}{2}\sum_{x\in \Lambda_n} \sigma_x,\frac{1}{2}\sum_{x\in S} \sigma_x\Big)+2\textup{Cov}(\mathbbm{1}_A,\mathbbm{1}_B).
\end{align*}
Combining the above displayed equations, we obtain 
\begin{equation*}
|\textup{Cov}(\mathbbm{1}_A,\mathbbm{1}_B)|\leq \frac{1}{4}\textup{Cov}\Big(\sum_{x\in \Lambda_n} \sigma_x,\sum_{x\in S} \sigma_x\Big),    
\end{equation*}
from which it is possible to conclude thanks to Corollary \ref{coro: truncated}.
\end{proof}

\begin{proof}[Proof of Corollary~\textup{\ref{cor: perco ratio mixing}}]
We first prove the result for random currents and for the random cluster measure. Fix $N>2n$. Let $\m_{in}\in \N^{E(\Lambda_n)}$ and $\m_{out}\in \N^{E(\Lambda_N)\setminus E(\Lambda_{2n})}$ be arbitrary current configurations with sources $A:=\partial \m_{in}\subset \partial\Lambda_n$ and $B:=\partial \m_{out}\subset \partial\Lambda_{2n}$. Below, we write $\mathbf P=\mathbf{P}^\emptyset_{\Lambda_N[\beta]}$ and abuse slightly notation by writing $Z^A_E=Z^A_G$ for a weighted graph $G=(V,E,J)$. Note that
\begin{align}
    \label{eq:mix_current_1}
    &\frac {\mathbf{P}[\n_{|E(\Lambda_n)}=\m_{in} ,\, \n_{|E(\Lambda_N)\setminus E(\Lambda_{2n})}=\m_{out}]}{\mathbf{P}[\n_{|E(\Lambda_n)}=\m_{in}]\mathbf{P}[\n_{|E(\Lambda_N)\setminus E(\Lambda_{2n})}=\m_{out}]} 
    = \frac{Z^{A\cup B}_{E(\Lambda_{2n})\setminus E(\Lambda_n)} Z^\emptyset_{\Lambda_N[\beta]}}{Z^{A}_{E(\Lambda_N)\setminus E(\Lambda_n)} Z^{B}_{\Lambda_{2n}[\beta]}} \\
    \label{eq:mix_current_2}
    &~~~~~~~~~~~~~~~~~~~~~~~~~~~~~~~~~~~~~~= \frac{ \sum_{\tilde \m_{in},\, \tilde\m_{out}} \, w(\tilde\m_{in})w(\tilde\m_{out}) Z_{E(\Lambda_{2n})\setminus E(\Lambda_n)}^{A\cup B} Z_{E(\Lambda_{2n})\setminus E(\Lambda_n)}^{\tilde A\cup \tilde B} }{ \sum_{\tilde \m_{in},\, \tilde\m_{out}} \, w(\tilde\m_{in})w(\tilde\m_{out}) Z_{E(\Lambda_{2n})\setminus E(\Lambda_n)}^{A\cup \tilde B} Z_{E(\Lambda_{2n})\setminus E(\Lambda_n)}^{\tilde A\cup B} },
\end{align}
where in the last line we decomposed $Z^\emptyset_{\Lambda_N[\beta]}$,  $Z^{A}_{E(\Lambda_N)\setminus E(\Lambda_n)}$ and $Z^{B}_{\Lambda_{2n}[\beta]}$ according to the restrictions $\tilde \m_{in}$ and $\tilde\m_{out}$ of the currents to
$E(\Lambda_n)$ and $E(\Lambda_N)\setminus E(\Lambda_{2n})$, with sources $\tilde A:=\partial \tilde\m_{in}\subset \partial\Lambda_n$ and $\tilde B:=\partial \tilde \m_{out}\subset \partial\Lambda_{2n}$. By the switching lemma and the proof of Corollary~\ref{coro: truncated} (see also Remark~\ref{rem: free even mixing}), we have
\begin{equation*}
    0\leq \frac{Z_{E(\Lambda_{2n})\setminus E(\Lambda_n)}^{A\cup B}Z_{E(\Lambda_{2n})\setminus E(\Lambda_n)}^{\emptyset}}{Z_{E(\Lambda_{2n}) \setminus E(\Lambda_n)}^{A} Z_{E(\Lambda_{2n})\setminus E(\Lambda_n)}^{ B}} -1 
    = \frac{ \langle\sigma_A;\sigma_B\rangle_{ E(\Lambda_{2n})\setminus E(\Lambda_n)} }{\langle\sigma_A\rangle_{E(\Lambda_{2n})\setminus E(\Lambda_n)} \langle\sigma_B\rangle_{E(\Lambda_{2n})\setminus E(\Lambda_n)} } \leq e^{-cn},
\end{equation*}
uniformly on the even sets $A\subset \partial\Lambda_n$ and $B\subset\partial\Lambda_{2n}$. In particular, we have
\begin{equation*}
    \left| \frac{Z_{E(\Lambda_{2n})\setminus E(\Lambda_n)}^{A\cup B} Z_{E(\Lambda_{2n}) \setminus E(\Lambda_n)}^{\tilde A\cup \tilde B}}{Z_{E(\Lambda_{2n})\setminus E(\Lambda_n)}^{A\cup \tilde B} Z_{E(\Lambda_{2n})\setminus E(\Lambda_n)}^{\tilde A\cup B}} - 1 \right| \leq e^{-cn}.
\end{equation*}
Plugging this bound into \eqref{eq:mix_current_2}, we conclude that 
\begin{equation*}
    \left| \frac{\mathbf{P}[\n_{|E(\Lambda_n)}=\m_{in} ,\, \n_{|E(\Lambda_N)\setminus E(\Lambda_{2n})}=\m_{out}]}{\mathbf{P}[\n_{|E(\Lambda_n)}=\m_{in}]\mathbf{P}[\n_{|E(\Lambda_N)\setminus E(\Lambda_{2n})}=\m_{out}]}  - 1 \right| \leq e^{-cn}.
\end{equation*}
Exponential ratio mixing follows for $\mathbf{P}=\mathbf{P}^\emptyset_{\Lambda_N[\beta]}$ by summing over $\m_{in}\in E$ and $\m_{out}\in F$, and then for $\mathbf{P}_\beta$ by taking $N\to\infty$.

Since $\phi_\beta$ is a Bernoulli sprinkling of $\mathbf{P}_\beta$ (see e.g.~\cite{LupuWerner2016note,AizenmanDuminilTassionWarzelEmergentPlanarity2019}), exponential ratio mixing for the former follows readily from the exponential ratio mixing of the latter.
\end{proof}

\bibliographystyle{alpha}
\bibliography{biblio}
\end{document}